\documentclass{elsart}

\usepackage{amssymb,amsmath,amscd, mathrsfs, msc}
\usepackage{cases}
\usepackage[all, 2cell,dvips]{xy}\UseAllTwocells\SilentMatrices
\usepackage{verbatim}
\usepackage[a]{esvect}
\usepackage{ntheorem}
\theoremstyle{plain}

\newtheorem{theorem}{Theorem}[section]
\newtheorem{lemma}[theorem]{Lemma}
\theoremstyle{definition}

\theoremstyle{remark}

\numberwithin{equation}{section} \theoremstyle{corollary}

\newenvironment{proof}[1][Proof]{\begin{trivlist}
\item[\hskip \labelsep {\bfseries #1}]}{\end{trivlist}}

\begin{document}
\begin{frontmatter}

\title{Analyticity of Bounded Solutions of Analytic State-Dependent Delay Differential  Equations 
}
 
\author[]{Qingwen Hu}\ead{\tt qingwen@utdallas.edu}
\address[]{Department of Mathematical Sciences,  The University of Texas at Dallas, Richardson, TX, 75080}

 \small{\today}


\maketitle
 
\begin{abstract}
We study the analyticity of bounded solutions of systems of analytic state-dependent delay differential equations. We obtain the analyticity of solutions by transforming the system of  state-dependent delay equations into an abstract ordinary differential equation in a subspace of the sequence space $l^{\infty}(\mathbb{R}^{N+1})$ and prove the existence of complex extension of the bounded solutions. An example is given to illustrate the general results.
\end{abstract}


 \end{frontmatter}

\thispagestyle{plain}

 
\pagestyle{plain}

\section{Introduction}\label{SOPS-4-1}

The  analyticity of bounded solutions of delay differential equations with constant delay such as the well-known Wright's equation was established in work of  Nussbaum \cite{Nussbaum-analyticity}. It is natural  to conjecture that this analyticity result holds true for many differential equations with state-dependent delay  such as
\begin{align}\label{eqn-2-oldd}\left\{
 \begin{aligned}
  \dot{x}(t)&=  f(x(t),\,x(t- \tau)),\\
   \tau & =r(x(t)),
 \end{aligned}
\right.
\end{align}
with analytic $f$ and $r$. In this paper, we solve this conjecture. 

We should  remark that the work of   Mallet-Paret and  Nussbaum \cite{Nussbaum-2} also presented some examples where bounded solutions are no-longer analytic, while  Krisztin \cite{Tibor} showed that globally defined bounded solutions of threshold type delay equations are analytic.   
Then an important theoretical problem is what would be the most general form of state-dependent delay differential equations for which the conjecture remains true for differential equations with state-dependent delay. We also notice that establishing the analyticity of bounded solutions such as periodic solutions is essential for describing the global dynamics of some state-dependent delay differential equations. For example, in \cite{HWZ} the nonexistence of a nonconstant $p$-periodic real-valued solution which is constant in a small interval in $\mathbb{R}$ was assumed in order to obtain the global continuation of periodic solutions of the following system
\begin{align}\label{eqn-2-old}\left\{
 \begin{aligned}
  \dot{x}(t)=&  f(x(t),\,x(t- \tau(t))),\\
   \dot{\tau}(t)=&   g(x(t), \tau (t)),
 \end{aligned}
\right.
\end{align}
with analytic $f$ and $g$. Specifically, it was needed to exclude the case where there is a nonconstant $p-$periodic solution for which 
\begin{align}
 \tau(t)=\tau_0,\,t\in I+kp,  k\in\mathbb{Z} 
\end{align} 
where $\tau_0>0$ is a constant and $I$ is an interval in $\mathbb{R}$ with length less than $p$. On the one hand, if there is such a periodic solution and if this solution is analytic on $\mathbb{R}$, then the delay $\tau$ must be a constant on the whole real line $\mathbb{R}$. On the other hand, under certain technical conditions, it can be ruled out the existence of such a periodic solution with constant delay by considering a cyclic system of ordinary differential equations (see \cite{HWZ} for more details) and hence these technical conditions can ensure the  nonexistence of a nonconstant $p$-periodic solution for which $\tau $ remains to be a constant in a small interval in $\mathbb{R}$.

In this paper, we first note that bounded solutions of system (\ref{eqn-2-oldd}) and system (\ref{eqn-2-old}) and many others including those with ``threshold delay" must satisfy the following differential equations with state-dependent delays
\begin{align}\label{eqn-2}\left\{
 \begin{aligned}
  \dot{x}(t)=&  f(x(t),\,x(t- \tau(t))),\\
   \dot{\tau}(t)=&   g(x(t),\,x(\eta(t)),\,\cdots,\,x(\eta^{M-1}(t)),\,\,\tau(t)),
 \end{aligned}
\right.
\end{align} where 
 $\eta^0(t)=t$, $\eta(t)=t-\tau(t)$,  $\eta^j(t)=\eta(\eta^{j-1}(t))$ for $j=1,\,2\,\,\cdots,\,M$ with $M\in\mathbb{N}$, and we assume 
\begin{description}
\item[(A1)]The maps $f$:
$U\times U\ni
(\theta_1,\theta_2) \rightarrow
f(\theta_1,\theta_2)\in\mathbb{C}^N$ and $g$: 
$U^M\times V \ni
(\gamma_1,\,\gamma_2)\rightarrow
g(\gamma_1,\,\gamma_2)\in\mathbb{C}$ 
are analytic with respect to $(\theta_1,\theta_2)$ and $(\gamma_1,\,\gamma_2)$, respectively, where $U\subset \mathbb{C}^N,\, V\subset\mathbb{C}$ are bounded open sets, $U^M=\underbrace{U\times U\times\cdots\times U}_{M}$.
 \item[(A2)] There exist   $l\in (0,\,1)$ and $c>1$  such that
$  |1-g(\gamma_1,\,\gamma_2)-\frac{c+l}{2}|<\frac{c-l}{2}$ for all
$(\gamma_1,\,\gamma_2)\in \overline{U}^M \times \overline{V}$, where $\overline{U}^M \times \overline{V}$ is the closure of $U^M\times V$.
\end{description}
(A1) is a natural assumption on the analyticity of  $f$ and $g$ on their domains.  (A2) is assuming that $g$ satisfies $l<|1-g|<c$ which ensures that  the mapping $\mathbb{R}\ni t\rightarrow t-\tau(t)\in\mathbb{R}$ is increasing with a bounded rate.

Let $(x,\,\tau)\in C(\mathbb{R};\mathbb{R}^{N+1})$ be a bounded solution  of  system~(\ref{eqn-2}) and define the sequence $\left((y_1,\,z_1),\,(y_2,\,z_2,),\,\cdots\right)$ by 
\begin{align*}(y_j(t),\, z_j(t))=\left(\frac{1}{c^{j}}x(\eta^{j-1}(t)),\,\frac{1}{c^{j}}\tau(\eta^{j-1}(t))\right)
\mbox{ for $j\geq 1,\,j\in\mathbb{N},\,t\in\mathbb{R}$.}  
\end{align*} The reason that we carry a term $\frac{1}{c^{j}}$ will be clear by the end of this section.
For $j=1$  we have for every $t\in\mathbb{R}$,
\begin{align}
 \frac{d}{dt}y_1(t)& = \dot{x}(t) =\frac{1}{c}{f(cy_1(t),\, c^2y_{2}(t)  )},\label{analyticity-ODE-1-J} \\  
\frac{d}{dt}z_1(t)& =\dot{\tau}(t) =\frac{1}{c}g(cy_1(t),\,c^2y_2(t),\,\cdots,\,c^{j+M-1}y_{j+M-1}(t),\, cz_{1}(t)  ).\label{analyticity-ODE-2-J}
\end{align}
For $j\geq 2,\,j\in\mathbb{N}$,  we have for every $t\in\mathbb{R}$,
\begin{align}
 \frac{d}{dt}y_j(t)& =\frac{1}{c^j}\dot{x}(\eta^{j-1}(t))\prod_{i=0}^{j-2}\dot{\eta}(\eta^i(t))\notag\\
                   & =\dot{x}(\eta^{j-1}(t))\frac{1}{c^j} \prod_{i=0}^{j-2}(1-g(x(\eta^i(t)), \,x(\eta^{i+1}(t)),\,\cdots,\,x(\eta^{i+M-1}(t)),\,\tau(\eta^i(t))  )\notag\\
                   & =\frac{f((c^{j}y_j(t),\, {c^{j+1}}y_{j+1}(t)  )}{1-g(c^{j}y_j(t), \,c^{j+1}y_{j+1}(t),\,\cdots,\,c^{j+M-1}y_{j+M-1}(t),\,c^{j}z_j(t)) }\notag\\
                   & \quad \times\frac{1}{c^j} \prod_{i=0}^{j-1}(1-g(c^{i+1}y_{i+1}(t), \,c^{i+2}y_{i+2}(t),\,\cdots,\,c^{i+M}y_{i+M}(t),\,c^{i+1}z_{i+1}(t))),\label{analyticity-ODE-1}
\intertext{and}
\frac{d}{dt}z_j(t)& =\frac{1}{c^j}\dot{\tau}(\eta^{j-1}(t))\prod_{i=0}^{j-2}\dot{\eta}(\eta^i(t))\notag\\
                   & =\dot{\tau}(\eta^{j-1}(t))\frac{1}{c^j} \prod_{i=0}^{j-2}(1-g(x(\eta^i(t)), \,x(\eta^{i+1}(t)),\,\cdots,\,x(\eta^{i+M-1}(t)),\,\tau(\eta^i(t))  )\notag\\
                   & =\frac{g(c^{j}y_j(t), \,c^{j+1}y_{j+1}(t),\,\cdots,\,c^{j+M-1}y_{j+M-1}(t),\,c^{j}z_j(t))}{1-g(c^{j}y_j(t), \,c^{j+1}y_{j+1}(t),\,\cdots,\,c^{j+M-1}y_{j+M-1}(t),\,c^{j}z_j(t))}\notag\\
                   & \quad \times\frac{1}{c^j} \prod_{i=0}^{j-1}(1-g(c^{i+1}y_{i+1}(t), \,c^{i+2}y_{i+2}(t),\,\cdots,\,c^{i+M}y_{i+M}(t),\,c^{i+1}z_{i+1}(t))).\label{analyticity-ODE-2}
\end{align}
Then the sequence $((y_1,\,z_1),\,(y_2,\,z_2),\,\cdots)$, $t\in\mathbb{R}$ satisfies a system of ordinary differential equations of  (\ref{analyticity-ODE-1-J}), (\ref{analyticity-ODE-2-J}), (\ref{analyticity-ODE-1}) and  (\ref{analyticity-ODE-2}).     Namely, for every $t\in\mathbb{R}$ and for $j\geq 1$,  we have
\begin{align} \label{New-Eqn-1}
\left\{
\begin{aligned}
\frac{d}{dt}y_j(t) & =\frac{f((c^{j}y_j(t),\, {c^{j+1}}y_{j+1}(t)  )}{1-g(c^{j}y_j(t), \,c^{j+1}y_{j+1}(t),\,\cdots,\,c^{j+M-1}y_{j+M-1}(t),\,c^{j}z_j(t)) } \\
                  & \quad \times\frac{1}{c^j} \prod_{i=0}^{j-1}(1-g(c^{i+1}y_{i+1}(t), \,c^{i+2}y_{i+2}(t),\,\cdots,\,c^{i+M}y_{i+M}(t),\,c^{i+1}z_{i+1}(t))),\\
\frac{d}{dt}z_j(t)  & =\frac{g(c^{j}y_j(t), \,c^{j+1}y_{j+1}(t),\,\cdots,\,c^{j+M-1}y_{j+M-1}(t),\,c^{j}z_j(t))}{1-g(c^{j}y_j(t), \,c^{j+1}y_{j+1}(t),\,\cdots,\,c^{j+M-1}y_{j+M-1}(t),\,c^{j}z_j(t))} \\
                    & \quad \times\frac{1}{c^j} \prod_{i=0}^{j-1}(1-g(c^{i+1}y_{i+1}(t), \,c^{i+2}y_{i+2}(t),\,\cdots,\,c^{i+M}y_{i+M}(t),\,c^{i+1}z_{i+1}(t))). 
 \end{aligned}\right. 
\end{align}
It remains to decide an appropriate space where $(y_1(t),\,z_1(t),\,(y_2,\,z_2),\,\cdots)$, $t\in\mathbb{R}$ lives in. With the arguments of $f$ and $g$ in the right hand side of (\ref{New-Eqn-1}), it turns out that we can set $w(t)=((y_1(t),\,z_1(t)),\,(y_2(t),\,z_2(t)),\,\cdots)$ to be in the sequence space $l_c^{\infty} (\mathbb{R}^{N+1})$  defined by
\begin{align}\label{l-c-def}
l_c^{\infty} (\mathbb{R}^{N+1})=\{v=(v_1,\,v_2,\,\cdots,v_j,\cdots)\in l^{\infty}(\mathbb{R}^{N+1}): \sup_{j\in\mathbb{N}} c^j|v_j|<+\infty\}, 
\end{align} where we can find a subset such that the terms of $f$ and $g$ in system~(\ref{New-Eqn-1}) are well-defined. Besides, the product terms
 in system~(\ref{New-Eqn-1}) need to be treated so that the right hand side of system~(\ref{New-Eqn-1}) always remains bounded as $j\rightarrow\infty$. We address this issue at Lemma~\ref{Lemma-l-infty}.
 
 With the above preparations we can represent   system~(\ref{New-Eqn-1})   by the following abstract ordinary differential equation:
\begin{align}\label{Abs-ODE-S-infty-new}
 \frac{d}{dt}w(t)= H(Tw(t)), 
\end{align}
where 
 the mapping $T:  l_c^{\infty} (\mathbb{R}^{N+1})\rightarrow l^{\infty}(\mathbb{R}^{N+1})$ 
is defined by
\[
 T(v_1,\,v_2,\,\cdots,\,v_j,\,\cdots)=(cv_1,\,c^2v_2,\,\cdots,c^{j}v_j,\cdots),
\] and $H:  l^{\infty} (\mathbb{R}^{N+1})\rightarrow l^{\infty}(\mathbb{R}^{N+1})$  is defined by the right hand side of   system~(\ref{New-Eqn-1}).

To obtain the analyticity of bounded solutions $(x(t),\,\tau(t))$, $t\in\mathbb{R}$ of  system~(\ref{eqn-2}), we follow the idea of \cite{Nussbaum-analyticity} to show that the solution $w(t)$ to  system~(\ref{Abs-ODE-S-infty-new}) has a complex extension and hence $(x(t),\,\tau(t))$, $t\in\mathbb{C}$ satisfies  system~(\ref{eqn-2}) on the complex domain. We remark that there are significant new challenges not present in \cite{Nussbaum-analyticity} but in this paper. First,  the operator $T$ is not a self mapping on $l_c^{\infty} (\mathbb{C}^{N+1})$ and the range of $H$ is in $l^{\infty} (\mathbb{R}^{N+1})$. This means that the right hand side of system~(\ref{Abs-ODE-S-infty-new}) does not define a vector field on $l_c^{\infty} (\mathbb{C}^{N+1})$ while we are looking for solutions in $l_c^{\infty} (\mathbb{C}^{N+1})$; Secondly, when we  transform system~(\ref{Abs-ODE-S-infty-new}) into an integral form and consider the associated fixed point problem on $l^{\infty} (\mathbb{C}^{N+1})$ using the uniform contraction principle in Banach spaces, we can not obtain a contractive mapping  on $l^{\infty} (\mathbb{C}^{N+1})$ unless we introduce a small perturbation. The problem is then reduced to show that the solution of the initial value problem associated with system~(\ref{Abs-ODE-S-infty-new}) is the limit of that of the perturbed system.

We  organize the remaining part of the paper as follows:  in section~\ref{preliminary}, we will develop results on analyticity of $H$ in the right hand side of  system~(\ref{Abs-ODE-S-infty-new}) and some basic functional analysis necessary for proving the existence of  complex extension of solutions to  system~(\ref{Abs-ODE-S-infty-new}), using the the uniform contraction principle in Banach spaces; 
We  present the main results in section~\ref{Main-results}
and will illustrate this general result with an example in the last section.

\section{Notations and Preliminary Results}\label{preliminary}

Let $E$ be a complex Banach space, $D$ an open subset of the complex plane $\mathbb{C}$. A continuous mapping $u: D\ni t\rightarrow u(t)\in E $ is called analytic if for every $t\in D$, 
$
 \lim_{t\rightarrow t_0} \frac{u(t)-u(t_0)}{t-t_0}=u'(t_0)
$ exists. If $W$ is an open subset of $E$, $\tilde{E}$ is a complex Banach space, a continuous mapping $G: W \ni u\rightarrow G(u)\in \tilde{E}$ is called analytic if for all $u_0\in W$, and for all $h\in E$, the mapping $t\rightarrow G(u_0+th)$ is analytic in the neighbourhood of $0\in \mathbb{C}$. 

Let $\mathbb{K}$ stand for the space of real numbers ($\mathbb{R}$) or complex numbers ($\mathbb{C}$). In the following, we develop some basic properties of the map $T$ and the spaces $l_c^{\infty} (\mathbb{K}^{N+1})$ and $l^{\infty} (\mathbb{K}^{N+1})$. We denote by $(v_j)_{j=1}^\infty$ the element $(v_1,\,v_2,\,\cdots,v_j,\cdots) $ in the sequence spaces.

\begin{lemma}\label{Banach-S-space}
 Let $c>1$ be a constant and $l_c^{\infty} (\mathbb{K}^{N+1})$  be defined by 
\[l_c^{\infty} (\mathbb{K}^{N+1})=\{v=(v_j)_{j=1}^\infty\in l^{\infty}(\mathbb{K}^{N+1}): \sup_{j\in\mathbb{N}} c^j|v_j|<+\infty\}. \] Then $l_c^{\infty} (\mathbb{K}^{N+1})$ is a Banach space   under the norm \mbox{$\|\cdot\|_{l_c^{\infty} (\mathbb{K}^{N+1})}$} defined by
\[
 \|v\|_{l_c^{\infty} (\mathbb{K}^{N+1})}=\sup_{j\in\mathbb{N}} c^j|v_j|.
\]  \end{lemma}

%
%
%

\begin{lemma}\label{l-m-space}
 Let $m\in\mathbb{N}, m\geq 2$ be a constant and $l_m^{\infty} (\mathbb{K}^{N+1})$  be defined by 
\[l_m^{\infty} (\mathbb{K}^{N+1})=\{v=(v_j)_{j=1}^\infty\in l^{\infty}(\mathbb{K}^{N+1}): \sup_{j\in\mathbb{N}} j^m|v_j|<+\infty\}. \] Then $l_m^{\infty} (\mathbb{K}^{N+1})$ is a Banach space   under the norm \mbox{$\|\cdot\|_{l_m^{\infty} (\mathbb{K}^{N+1})}$} defined by
\[
 \|v\|_{l_m^{\infty} (\mathbb{K}^{N+1})}=\sup_{j\in\mathbb{N}} j^m|v_j|.
\] 
Moreover, the embedding $I_m: l_m^{\infty} (\mathbb{K}^{N+1})\rightarrow l^{\infty} (\mathbb{K}^{N+1})$ is compact.
 \end{lemma}
\begin{proof}
 It is clear that $l_m^{\infty} (\mathbb{K}^{N+1})$ is a subspace of $l^{\infty}(\mathbb{K}^{N+1})$ and  $\|v\|_{l_m^{\infty} (\mathbb{K}^{N+1})}=\sup_{j\in\mathbb{N}} j^m|v_j|$ defines a norm on $l_m^{\infty} (\mathbb{K}^{N+1})$. Let $\{v^n\}_{n=1}^{\infty}$  be a Cauchy sequence in $l_m^{\infty} (\mathbb{K}^{N+1})$. For every $n\in\mathbb{N}$, let  $b^n=(v_1^{n},\,2^mv_2^{n},\,\cdots,j^mv_j^{n}, \cdots)$. Then   $\{b^n\}_{n=1}^{\infty}$  is a Cauchy sequence in $l^{\infty}(\mathbb{K}^{N+1})$. Since $l^{\infty}(\mathbb{K}^{N+1})$ is a Banach space, there exists $b^*\in l^{\infty}(\mathbb{K}^{N+1})$ so that
\[
 \lim_{n\rightarrow+\infty}|b^n-b^*|_{l^{\infty}(\mathbb{K}^{N+1})}=0.
\]
Then we have $v^*=(\frac{b_1^{*}}{1},\,\frac{b_2^{*}}{2^m},\,\cdots,\frac{b_j^{*}}{j^m}\cdots)\in l_m^{\infty} (\mathbb{K}^{N+1})$ and
\[
  \lim_{n\rightarrow+\infty}|v^n-v^*|_{l_m^{\infty} (\mathbb{K}^{N+1})}=\lim_{n\rightarrow+\infty}|b^n-b^*|_{l^{\infty}(\mathbb{K}^{N+1})}=0.\]
Next we show that the embedding $I_m: l_m^{\infty} (\mathbb{K}^{N+1})\rightarrow l^{\infty} (\mathbb{K}^{N+1})$ is compact.
  For every $k\in\mathbb{N}$ we define the ``cut-off'' operator $H_k:  l_m^{\infty} (\mathbb{K}^{N+1})\rightarrow l^{\infty} (\mathbb{K}^{N+1})$ 
  by
\[
 H_k(v_1,\,v_2,\,\cdots,\,v_k,\,\cdots)=(v_1,\,v_2,\,\cdots, v_k,\,0,\cdots).
\] Then $H_k$ is compact since the dimension of the range is finite. Moreover we have
\[
\|(I_m-H_k)(v_1,\,v_2,\,\cdots,\,v_j,\,\cdots)\|_{_{l^{\infty} (\mathbb{K}^{N+1})}}=\sup_{j\geq k+1}|v_j|,
\]
which implies that $\|I_m-H_k\|\rightarrow 0$ as $k\rightarrow +\infty$ and hence $I_m$ is compact.
\hfill\hspace*{1em}\qed\end{proof}

 \begin{lemma}\label{Banach-spaces}Let $c>1$ be a constant. The closed unit ball of
  $l_c^{\infty} (\mathbb{K}^{N+1})$ is closed under the norm $\|\cdot\|_{l^{\infty}(\mathbb{K}^{N+1})}$.  
\end{lemma}
\begin{proof} Let $B_c(1)=\{v\in l_c^{\infty}(\mathbb{K}^{N+1}): \|v\|_{l_c^{\infty}(\mathbb{K}^{N+1})}\leq 1\}.$  Let $\{v^n\}_{n=1}^{+\infty}\subset B_c(1)$ be a Cauchy sequence in the norm $\|\cdot\|_{l^{\infty}(\mathbb{K}^{N+1})}$. Since  $l_c^{\infty} (\mathbb{K}^{N+1})$ is a subspace of the Banach space $(l^{\infty} (\mathbb{K}^{N+1}),\,|\cdot|_{l^{\infty}(\mathbb{K}^{N+1})})$. There exists $v^0\in l^{\infty} (\mathbb{K}^{N+1})$ such that
\begin{align}\label{Bs-0}
\lim_{n\rightarrow +\infty}\|v^n-v^0\|_{l^{\infty}(\mathbb{K}^{N+1})}=0.                                                                                                                                                                                                                                                                                               \end{align}Now we show that $v^0\in B_c(1)$. By way of contradiction, assume that $v^0\not\in B_c(1)$. Then we distinguish the following two cases:\\
\textit{Case 1}. $v^0\not\in l_c^{\infty}(\mathbb{K}^{N+1})$.
 Then for every $K>0$,  there exists $j_0\in\mathbb{N}$ such that $c^{j_0}|(v^0)_{j_0}|>K$. That is,
\begin{align}\label{Bs-1}
|(v^0)_{j_0}|>\frac{K}{c^{j_0}}.
\end{align}
On the other hand, it follows from
(\ref{Bs-0}) that
for every $\epsilon>0$, there exists $N_0\in\mathbb{N}$ such that for every $n>N_0$, we have
$
  \sup_{j\in\mathbb{N}}|(v^0)_{j}-(v^n)_{j}|<\epsilon 
$
which leads to
$
 |(v^0)_{j}|- |(v^n)_{j}|<\epsilon,\,\mbox{for every $j\in\mathbb{N}$}, n>N_0. 
$
It follows that
\begin{align}\label{Bs-tri-inequality-0}
 |(v^n)_{j}|>  |(v^0)_{j}|-\epsilon,\,\mbox{for every $j\in\mathbb{N}$}, n>N_0. 
\end{align}
Choosing $j=j_0$ and $\epsilon= \frac{K}{2c^{j_0}}$ in (\ref{Bs-tri-inequality-0}), then by (\ref{Bs-0}) and (\ref{Bs-1})  we obtain that
$
  |(v^n)_{j_0} |  \geq   |(v^0)_{j_0}|-\frac{K}{2c^{j_0}}  >\frac{K}{2c^{j_0}},
$ which leads to
$  |c^{j_0}(v^n)_{j_0}|>K/2$ for every $n>N_0$. That is, $\lim_{n\rightarrow+\infty}\|v^n\|_{l_c^{\infty}(\mathbb{K}^{N+1})}=+\infty.$ 
   This is a contradiction since $\{v^n\}_{n=1}^{+\infty}\subset B_c(1)$. \\
\textit{Case 2}. $v^0\in l_c^{\infty}(\mathbb{K}^{N+1})$ but $\|v^0\|_{l_c^{\infty}(\mathbb{K}^{N+1})}>1$. Let $s=\|v^0\|_{l_c^{\infty}(\mathbb{K}^{N+1})}$. Then $s>1$ and there exists $j_1\in\mathbb{N}$ such that $c^{j_1}|(v^0)_{j_1}|>1$. That is,
\begin{align}\label{Bs-2}
\frac{s}{c^{j_1}}=|(v^0)_{j_1}|>\frac{1}{c^{j_1}}.
\end{align}
On the other hand, it follows from
(\ref{Bs-0}) that
for every $\epsilon>0$, there exists $N_1\in\mathbb{N}$ such that for every $n>N_1$, we have
$
  \sup_{j\in\mathbb{N}}|(v^n)_{j}-(v^0)_{j}|<\epsilon 
$
which leads to
$
 |(v^0)_{j}|- |(v^n)_{j}|<\epsilon,\,\mbox{for every $j\in\mathbb{N}$}, n>N_1. 
$
It follows that
\begin{align}\label{Bs-tri-inequality}
 |(v^n)_{j}|>  |(v^0)_{j}|-\epsilon,\,\mbox{for every $j\in\mathbb{N}$}, n>N_0. 
\end{align}
Note that $\{v^n\}_{n=1}^{+\infty}\subset B_c(1))$. Then by (\ref{Bs-tri-inequality}) we have                                                                                                   \begin{align}\label{Bs-tri-inequality-1}
 \frac{1}{c^j}\geq |(v^n)_{j}|>  |(v^0)_{j}|-\epsilon,\,\mbox{for every $j\in\mathbb{N}$}, n>N_0. 
\end{align}
 Choosing $j=j_1$, $\epsilon=\frac{s-1}{2c^{j_1}}$ in  (\ref{Bs-tri-inequality-1}) we obtain from (\ref{Bs-2}) that
   \begin{align}\label{Bs-tri-inequality-2}
 \frac{1}{c^{j_1}}\geq |(v^n)_{j_1}|>  |(v^0)_{j_1}|-\epsilon=\frac{s}{c^{j_1}}-\frac{s-1}{2c^{j_1}},\,\mbox{for every }\, n>N_0. 
\end{align}                                                                                              
Then we have $s<1$. This is a contradiction.
\hfill\hspace*{1em}\qed \end{proof}

We remark that the unit sphere of $l_c^{\infty} (\mathbb{K}^{N+1})$ is not closed under the norm $\|\cdot\|_{l^{\infty}(\mathbb{K}^{N+1})}$.  
In light of Lemma~\ref{Banach-spaces} we will equip  bounded sets of $l_c^{\infty} (\mathbb{K}^{N+1})$  with the norm $\|\cdot\|_{l^{\infty}(\mathbb{K}^{N+1})}$.  The following three lemmas discuss the properties of a linear operator on $l_c^{\infty} (\mathbb{K}^{N+1})$ equipped with the norm $\|\cdot\|_{l^{\infty}(\mathbb{K}^{N+1})}$.

\begin{lemma}\label{OMT}Let $c>1$ be a constant. 
 The mapping $T:  (l_c^{\infty} (\mathbb{K}^{N+1},\,\|\cdot\|_{l^{\infty} (\mathbb{K}^{N+1})})\rightarrow (l^{\infty}(\mathbb{K}^{N+1}),\,\|\cdot\|_{l^{\infty}(\mathbb{K}^{N+1})})$ 
defined by
\[
 T(v_1,\,v_2,\,\cdots,\,v_j,\,\cdots)=(cv_1,\,c^2v_2,\,\cdots,c^{j}v_j,\cdots),
\]
has a compact inverse $T^{-1}$ with norm $\|T^{-1}\|=\frac{1}{c}$. Moreover, $T$ is a closed operator.
\end{lemma}
\begin{proof}  We first show that $T^{-1}$ exists and is continuous. By definition of $T$ and that $c>1$, we know that $T$ is 1-1 and onto. Therefore $T^{-1}:  (l^{\infty} (\mathbb{K}^{N+1},\,\|\cdot\|_{l^{\infty} (\mathbb{K}^{N+1})})\rightarrow (l_c^{\infty}(\mathbb{K}^{N+1}),\,\|\cdot\|_{l^{\infty}(\mathbb{K}^{N+1})})$  exists and is given by
\[
 T^{-1}(v_1,\,v_2,\,\cdots,\,v_j,\,\cdots)=(c^{-1}v_1,\,c^{-2}v_2,\,\cdots,c^{-j}v_j,\cdots).
\] Moreover, we have
\begin{align*}
 \|T^{-1}\|
     = & \sup_{v\in l^{\infty} (\mathbb{K}^{N+1})}\frac{\|T^{-1}v\|_{l^{\infty}(\mathbb{K}^{N+1})}}{\|v\|_{l^{\infty}(\mathbb{K}^{N+1})}}
     =  \sup_{\|v\|_{l^{\infty}(\mathbb{K}^{N+1})}=1}\|T^{-1}v\|_{l^{\infty}(\mathbb{K}^{N+1})}
    =  \frac{1}{c}.
\end{align*} 
Next we show that $T^{-1}$ is compact.  For every $m\in\mathbb{N}$ we define an operator $H_m:  (l^{\infty} (\mathbb{K}^{N+1},\,\|\cdot\|_{l^{\infty} (\mathbb{K}^{N+1})})\rightarrow (l_c^{\infty}(\mathbb{K}^{N+1}),\,\|\cdot\|_{l^{\infty}(\mathbb{K}^{N+1})})$ 
  by
\[
 H_m(v_1,\,v_2,\,\cdots,\,v_j,\,\cdots)=(c^{-1}v_1,\,c^{-2}v_2,\,\cdots, c^{-m}v_m,\,0,\cdots).
\] Then $H_m$ is compact since the dimension of the range is finite. Moreover we have
\[
\|(T^{-1}-H_m)(v_1,\,v_2,\,\cdots,\,v_j,\,\cdots)\|_{_{l^{\infty} (\mathbb{K}^{N+1})}}=\sup_{j\geq m+1}c^{-j}\|(v_1,\,v_2,\,\cdots,\,v_j,\,\cdots)\|_{_{l^{\infty} (\mathbb{K}^{N+1})}},
\]
which implies that $\|T^{-1}-H_m\|\rightarrow 0$ as $m\rightarrow +\infty$ and hence $T^{-1}$ is compact.

Next we show that $T$ is a closed operator. Let $\{v^n\}_{n=1}^\infty\subset {l_c^{\infty} (\mathbb{K}^{N+1})}$ be a convergent sequence such that $\lim_{n\rightarrow+\infty}\|v^n-v\|_{l^{\infty} (\mathbb{K}^{N+1})}=0$ for some $v\in l^{\infty} (\mathbb{K}^{N+1})$, and such that $\lim_{n\rightarrow+\infty}\|Tv^n-u\|_{l^{\infty} (\mathbb{K}^{N+1})}=0$ for some $u\in l^{\infty} (\mathbb{K}^{N+1})$. Then we have
\begin{align*}
\|T^{-1}u-v\|_{l^{\infty} (\mathbb{K}^{N+1})} & =\|T^{-1}u-v^n+v^n-v\|_{l^{\infty} (\mathbb{K}^{N+1})} \\
& =\|T^{-1}u-v^n\|_{l^{\infty} (\mathbb{K}^{N+1})} +\|v^n-v\|_{l^{\infty} (\mathbb{K}^{N+1})} \\
& \leq \|T^{-1}\|\cdot\|u-Tv^n\|_{l^{\infty} (\mathbb{K}^{N+1})} +\|v^n-v\|_{l^{\infty} (\mathbb{K}^{N+1})} \\
&\rightarrow 0 \mbox{ as } n\rightarrow+\infty.
\end{align*}
Therefore we have $T^{-1}u-v=0$. That is, $Tv=u$. $T$ is closed.
    \hfill\hspace*{0.05em}\qed \end{proof}

Denote by $\mathscr{L}(l^{\infty}(\mathbb{K}^{N+1});\, l^{\infty}(\mathbb{K}^{N+1}))$  the space of bounded linear operators from $l^{\infty}(\mathbb{K}^{N+1})$ to $l^{\infty}(\mathbb{K}^{N+1}))$. We have the following two lemmas which will be used when we deal with the integral forms of the relevant abstract ordinary differential equations.
 
\begin{lemma}\label{Compact-operator}
Let the mapping $T:  l_c^{\infty} (\mathbb{K}^{N+1})\rightarrow l^{\infty}(\mathbb{K}^{N+1})$ 
be as in Lemma~\ref{OMT} and $\lambda\geq 0$. Then   the mappings $I - T^{-1}$ and $\lambda I+ T^{-1}: l^{\infty}(\mathbb{K}^{N+1})\rightarrow l^{\infty}(\mathbb{K}^{N+1})$ are   bounded linear operators with  
\begin{align*}
 \| I - T^{-1}\|_{\mathscr{L}(l^{\infty}(\mathbb{K}^{N+1});\, l^{\infty}(\mathbb{K}^{N+1}))}&=1,\\ 
 \|  \lambda I + T^{-1}\|_{\mathscr{L}(l^{\infty}(\mathbb{K}^{N+1});\, l^{\infty}(\mathbb{K}^{N+1}))}&=\lambda+\frac{1}{c}. 
\end{align*}
Moreover, if $\lambda\in (0,\,1-1/c)$ then \[\| (1-\lambda) I - T^{-1}\|_{\mathscr{L}(l^{\infty}(\mathbb{K}^{N+1});\, l^{\infty}(\mathbb{K}^{N+1}))}=1-\lambda. \]
\end{lemma}
\begin{proof}
Let $S(1)=\{v\in l^{\infty}(\mathbb{K}^{N+1}):\sup_{j\in\mathbb{N}}|v_j|=1\}\subset l^{\infty}(\mathbb{K}^{N+1}).$ Note that  
\begin{align*}
 \|  I - T^{-1}\|_{\mathscr{L}(l^{\infty}(\mathbb{K}^{N+1});\, l^{\infty}(\mathbb{K}^{N+1}))}
=& \sup_{v\in S(1)} \sup_{j\in\mathbb{N}}(1-c^{-j})|v_j|\\
\leq &  \sup_{v\in S(1)} \left(\sup_{j\in\mathbb{N}}|v_j|-\inf_{j\in\mathbb{N}} c^{-j} |v_j|\right)\\
=& 1.
\end{align*}Taking $v_0=\{\frac{j}{j+1}\vec{e}\}_{j=1}^{\infty}\in S(1)$ where $\vec{e}$ is a unit vector on the boundary of the unit ball of $\mathbb{K}^{N+1}$, we have
\begin{align*}
 \|  I - T^{-1}\|_{\mathscr{L}(l^{\infty}(\mathbb{K}^{N+1});\, l^{\infty}(\mathbb{K}^{N+1}))}
=& \sup_{v\in S(1)} \sup_{j\in\mathbb{N}}(1-c^{-j})|v_j|\\
\geq &  \sup_{v=v_0} \left( \sup_{j\in\mathbb{N}}\frac{j}{j+1}(1-c^{-j}) \right)\\
=& 1.
\end{align*} It follows that $\|   I - T^{-1}\|_{\mathscr{L}(l^{\infty}(\mathbb{K}^{N+1});\, l^{\infty}(\mathbb{K}^{N+1}))}=1$. Moreover, we have
\begin{align*}
 \| \lambda I + T^{-1}\|_{\mathscr{L}(l^{\infty}(\mathbb{K}^{N+1});\, l^{\infty}(\mathbb{K}^{N+1}))}
=& \sup_{v\in S(1)} \sup_{j\in\mathbb{N}}(\lambda+c^{-j})|v_j|\\
\leq &  \sup_{v\in S(1)} \left(\lambda \sup_{j\in\mathbb{N}}|v_j|+\sup_{j\in\mathbb{N}} c^{-j} |v_j|\right)\\
=& \lambda+\frac{1}{c}.
\end{align*}Taking $v'_0=\{c^{-(j-1)}\vec{e}\}_{j=1}^{\infty}\in S(1)$, we have
\begin{align*}
 \|\lambda I + T^{-1}\|_{\mathscr{L}(l^{\infty}(\mathbb{K}^{N+1});\, l^{\infty}(\mathbb{K}^{N+1}))}
=& \sup_{v\in S(1)} \sup_{j\in\mathbb{N}}(\lambda+c^{-j})|v_j|\\
\geq &  \sup_{v=v'_0} \left( \sup_{j\in\mathbb{N}}c^{-(j-1)}(\lambda+c^{-j}) \right)\\
=& \lambda+\frac{1}{c}.
\end{align*} It follows that $\|\lambda I + T^{-1}\|_{\mathscr{L}(l^{\infty}(\mathbb{K}^{N+1});\, l^{\infty}(\mathbb{K}^{N+1}))}= \lambda+\frac{1}{c}$.  

Finally, we show that $\| (1-\lambda) I - T^{-1}\|_{\mathscr{L}(l^{\infty}(\mathbb{K}^{N+1});\, l^{\infty}(\mathbb{K}^{N+1}))}=1-\lambda.$ Note that we have $1-\lambda-c^{-j}>0$ for all $j\in\mathbb{N}$ since $\lambda\in (0,\,1-1/c)$. Then on the one hand we have
\begin{align*}
\| (1-\lambda) I - T^{-1}\|_{\mathscr{L}(l^{\infty}(\mathbb{K}^{N+1});\, l^{\infty}(\mathbb{K}^{N+1}))}=& \sup_{v\in B(1)} \sup_{j\in\mathbb{N}}(1-\lambda-c^{-j})|v_j|\\
\leq & \sup_{j\in\mathbb{N}}(1-\lambda-c^{-j})\\
=& 1-\lambda.
\end{align*}On the other hand,
\begin{align*}
\| (1-\lambda) I - T^{-1}\|_{\mathscr{L}(l^{\infty}(\mathbb{K}^{N+1});\, l^{\infty}(\mathbb{K}^{N+1}))}=& \sup_{v\in B(1)} \sup_{j\in\mathbb{N}}(1-\lambda-c^{-j})|v_j|\\
\geq &  \sup_{v=v'_0}\sup_{j\in\mathbb{N}}(1-\lambda-c^{-j})|v_j|\\
=&   \sup_{j\in\mathbb{N}}(1-\lambda-c^{-j})\frac{j}{j+1}\\
=& 1-\lambda.
\end{align*}It follows that $\|(1-\lambda) I - T^{-1}\|_{\mathscr{L}(l^{\infty}(\mathbb{K}^{N+1});\, l^{\infty}(\mathbb{K}^{N+1}))}=1-\lambda.$\hfill \qed
\end{proof}

\begin{lemma}\label{Extension-operator}
Let the mapping $T:  l_c^{\infty} (\mathbb{K}^{N+1})\rightarrow l^{\infty}(\mathbb{K}^{N+1})$ 
be as in Lemma~\ref{OMT}. Then for every $\lambda\geq 0$,    the mapping
 $(\lambda T + I)^{-1}: l^{\infty}(\mathbb{K}^{N+1})\rightarrow l_c^{\infty}(\mathbb{K}^{N+1})\subset l^{\infty}(\mathbb{K}^{N+1})$ is continuous with norm
\begin{align*}  
 \| (\lambda T + I)^{-1}\|_{\mathscr{L}(l^{\infty}(\mathbb{K}^{N+1});\, l^{\infty}(\mathbb{K}^{N+1}))}&=\frac{1}{c\lambda+1}.
\end{align*} 
\end{lemma}
\begin{proof} We compute $\|(\lambda T + I)^{-1}\|_{\mathscr{L}(l^{\infty}(\mathbb{K}^{N+1});\, l^{\infty}(\mathbb{K}^{N+1}))}$. Let $S(1)=\{v\in l^{\infty}(\mathbb{K}^{N+1}):\sup_{j\in\mathbb{N}}|v_j|=1\}\subset l^{\infty}(\mathbb{K}^{N+1}).$ Note that  
\begin{align*}
 \|(\lambda T + I)^{-1}\|_{\mathscr{L}(l^{\infty}(\mathbb{K}^{N+1});\, l^{\infty}(\mathbb{K}^{N+1}))}
=& \sup_{v\in S(1)} \sup_{j\in\mathbb{N}}\frac{|v_j|}{\lambda {c^j}+1}\\
=& \sup_{v\in S(1)} \sup_{j\in\mathbb{N}}\frac{|v_j|}{\lambda c^j +1}\\
\leq &  \sup_{j\in\mathbb{N}} \frac{1}{\lambda c^j +1}\\
=& \frac{1}{\lambda c +1}.
\end{align*} Taking $v_0=\{c^{-(j-1)}\vec{e}\}_{j=1}^{\infty}\in B_c(1)$, we have
\begin{align*}\|(\lambda T + I)^{-1}\|_{\mathscr{L}(l^{\infty}(\mathbb{K}^{N+1});\, l^{\infty}(\mathbb{K}^{N+1}))}=&\sup_{v\in S(1)} \sup_{j\in\mathbb{N}}\frac{|v_j|}{\lambda {c^j}+1}\\
\geq & \sup_{v=v_0} \sup_{j\in\mathbb{N}}\frac{1}{\lambda {c^j}+1}|v_j|\\
\geq &   \sup_{j\in\mathbb{N}}\frac{c^{-(j-1)}}{\lambda {c^j}+1} \\
=& \frac{1}{\lambda c +1}.
\end{align*} It follows that $\|(\lambda T + I)^{-1}\|_{\mathscr{L}(l^{\infty}(\mathbb{K}^{N+1});\, l^{\infty}(\mathbb{K}^{N+1}))}= \frac{1}{\lambda c +1}$. 
\hfill\hspace*{1em}\qed
\end{proof}

The following three lemmas address the well-posedness of system~(\ref{Abs-ODE-S-infty-new}) and the analyticity of the map $H$.

\begin{lemma}\label{Lemma-l-infty}Assume $\textrm{(A1)-(A2)}$. For every  sequence 
$\{(u_i,\,v_i)\}_{i=0}^{+\infty}\subset U\times V$, let $\mu_i=(u_i,\,u_{i+1},\,\cdots,\,u_{i+M-1},\,v_i)\in U^M\times V$. Then we have
\begin{align*}
 \lim_{j\rightarrow+\infty} \frac{1}{c^j}\prod_{i=0}^{j-1}|1-g(\mu_i)|=0,
\end{align*}
Moreover, for every $m\in\mathbb{N}$, we have
\[
 \lim_{j\rightarrow+\infty} \frac{j^m}{c^j} \prod_{i=0}^{j-1}|1-g(\mu_i)|=0.
\]
\end{lemma}
\begin{proof}   By (A2), we have  $|1-g(\gamma_1,\,\gamma_2)|<c$ and $|1-g(\gamma_1,\,\gamma_2)|$ with $(\gamma_1,\,\gamma_2)\in \overline{U}^M\times \overline{V}$ has a supremum less than $c$. Let $s>0$ be such that $c=e^s$. Then there exists $N_0\geq 1$, $N_0\in\mathbb{N}$ so that $|1-g(\mu_i)|\leq e^{s(1-\frac{1}{N_0})}$    for all $i\in\mathbb{N}$. Then for every $n\in\mathbb{N}$ we have
\[|1-g(\mu_i)|\leq  e^{s(1-\frac{1}{N_0})}\leq  e^{s(1-\frac{n}{i})}\textrm { for all }  i\geq nN_0. \] 
It follows that $\ln \left( \frac{|1-g(\mu_i)|}{c}\right)\leq -\frac{ns}{i} $ for all  $  i\geq nN_0$. Then for $j> nN_0$   we have
\begin{align}
 \sum_{i=0}^{j-1}\ln \left(\frac{|1-g(\mu_i )|}{c}\right)& =\sum_{i=0}^{nN_0-1}\ln \left(\frac{|1-g(\mu_i)|}{c}\right)+\sum_{i=nN_0}^{j-1}\ln \left(\frac{|1-g(\mu_i)|}{c}\right)\notag\\
& \leq \sum_{i=0}^{nN_0-1}\ln \left(\frac{|1-g(\mu_i)|}{c}\right)+s\sum_{i=nN_0}^{j-1} \left(-\frac{n}{i}\right).\label{analyticity-infty-new}
\end{align}
Let $c_0=\sum_{i=0}^{nN_0-1}\ln \left(\frac{|1-g(\mu_i)|}{c}\right)$. Then by (\ref{analyticity-infty-new}) and (A2), we have
\begin{align}
  0<\frac{1}{c^j}\prod_{i=0}^{j-1}|1-g(\mu_i)|& = \exp{\sum\limits_{i=0}^{j-1}\ln \left(\frac{|1-g(\mu_i)|}{c}\right)}\notag\\
& \leq  e^{c_0}\exp{\left(s\sum\limits_{i=nN_0}^{j-1} -\frac{n}{i}\right)}.\label{analyticity-infty-inequality}
\end{align}
Taking limits as $j\rightarrow +\infty$ in (\ref{analyticity-infty-inequality}) we have
\[
 \lim_{j\rightarrow+\infty} \frac{1}{c^j}\prod_{i=0}^{j-1}|1-g(\mu_i)|=0.
\]
Choosing $n=m$ in the inequality (\ref{analyticity-infty-inequality}), we have
\begin{align}
0< \frac{j^m}{c^j}  \prod_{i=0}^{j-1}|1-g(\mu_i)|& \leq j^m e^{c_0}\exp{\left(s\sum\limits_{i=mN_0}^{j-1} -\frac{m}{i} \right)}\notag\\
& = \exp\left(c_0+\sum\limits_{i=1}^{mN_0-1}\left(\frac{m}{i}\right)+\frac{m}{j}\right) \frac{j^m}{\exp{(mH_j)}}\notag\\
& = \exp\left(c_0+\sum\limits_{i=1}^{mN_0-1}\left(\frac{m}{i}\right)+\frac{m}{j}\right)\exp(m\ln j-mH_j),\label{analyticity-infty-inequality-2}
\end{align}
where $H_j=1+\frac{1}{2}+\cdots+\frac{1}{j}$ and $\sum_{i=1}^{mN_0-1}\left(\frac{m}{i}\right)$ is regarded 0 if $mN_0=1$. We note that $\lim_{j\rightarrow+\infty}\ln j-H_j=-\gamma$ where $\gamma>0$ is the Euler-M\'{a}scheroni constant. Taking supremum limits as $j\rightarrow +\infty$ in (\ref{analyticity-infty-inequality-2}) we have
\[
0<\limsup_{j\rightarrow+\infty}\frac{j^m}{c^j} \prod_{i=0}^{j-1}|1-g(\mu_i)|\leq \exp\left(c_0+\sum\limits_{i=1}^{mN_0-1}\left(\frac{m}{i}\right)-m\gamma\right)<+\infty.
\]Then we have
\begin{align*}
  & \lim_{j\rightarrow+\infty}\frac{j^{m-1}}{c^j} \prod_{i=0}^{j-1}|1-g(\mu_i)|\\
  \leq&  \limsup_{j\rightarrow+\infty}\frac{j^m}{c^j} \prod_{i=0}^{j-1}|1-g(\mu_i)|\lim_{j\rightarrow+\infty}\frac{1}{j}\\
   = &\,0.
\end{align*}
Since $m\in\mathbb{N}$ is arbitrary, it follows that $ \lim_{j\rightarrow+\infty} \frac{j^m}{c^j} \prod_{i=0}^{j-1}|1-g(\mu_i)|=0$ for all $m\in\mathbb{N}$. \hfill \qed
\end{proof}
 Let $ l^{\infty}(U\times V)$  be the subset of  $l^{\infty}(\mathbb{K}^{N+1})$ defined by
 \[
 l^{\infty}(U\times V)=\prod_{j=0}^\infty (U\times V).
 \]
Note that $ l^{\infty}(U\times V)$ is not an open set of $l^{\infty}(\mathbb{K}^{N+1})$ if $l^{\infty}(\mathbb{K}^{N+1})$ is equipped with the product topology. However,  we are concerned with the  following set:
\begin{align}\label{set-A}
A=\{w=(w_0,\,w_1,\,\cdots)\in l^{\infty}(U\times V): \mbox{$\{w_j\}_{j=0}^\infty\subset Q_0$ for some compact $Q_0\subset U\times V$} \}.
\end{align}
For every $w=(w_0,\,w_1,\,\cdots)\in A$, we can find an open set $P$ and a compact set $Q$ such that $\{w_j\}_{j=0}^\infty\subset P\subset Q\subset U\times V$. Then  $w\in l^\infty(P)\subset A\subset  l^{\infty}(U\times V)$. Namely, $A$ is open under the box topology.

 We also define the projections $\chi_i: l^{\infty}(U\times V)\rightarrow U^M\times V$ with $i\in\{0,\,1,\,2,\,\cdots\}$ by 
 \begin{align}\label{chi}
  \chi_i (w)=(u_i,\,u_{i+1},\,\cdots,\,u_{i+M-1},\,v_i)
 \end{align}for every $w=((u_i,\,v_i))_{j=1}^\infty\in  l^{\infty}(U\times V).$
\begin{lemma}\label{Lemma-G}Let $A$ be defined at (\ref{set-A}).
Assume $($\textrm{A1 -- A2}$\,)$.
 The mapping $G$ defined by
\[
 G: A\ni w=(w_0,\,w_1,\,w_2,\,\cdots,\,w_i,\,\cdots)\rightarrow G(w)=\left(\frac{1}{c^j}\prod_{i=0}^{j-1}(1-g(\chi_i (w)))\right)_{j=1}^{+\infty},
\]where $w_i=(u_i,\,v_i)\in U\times V$,  is   continuous and analytic from $A\subset l^{\infty}(U\times V)$ to $ l^{\infty}(\mathbb{C}^{N+1})$.
\end{lemma}
\begin{proof}  
 By Lemma~\ref{Lemma-l-infty}, we know that   $G$ is a mapping from $l^{\infty}(\mathbb{C}^{N+1})$ to   $  l^{\infty}(\mathbb{C}^{N+1})$. Note that for every $i,\,j\in\mathbb{N}$ with $0\leq i\leq j-1$, we have
\begin{align}\label{product-derivative}
 \frac{\partial }{\partial \mu_i}\prod_{i=0}^{j-1}(1-g(\mu_i  ))&=
 \frac{-\frac{\partial }{\partial  \mu_i}g(\mu_i)}{(1-g(\mu_i))}\prod_{i=0}^{j-1}(1-g(\mu_i  )).
\end{align}
Let $(\mu_0,\,\mu_1,\,\cdots,\,\mu_{j-1})$ denote a column vector in $\bigoplus\limits_{i=0}^{j-1}\mathbb{C}^{MN+1}$. Then we have
\begin{align*}
 &\frac{\partial }{\partial (\mu_0,\,\mu_1,\,\cdots,\,\mu_{j-1})}\prod_{i=0}^{j-1}(1-g(\mu_i  ))\\
=&\left(\prod_{i=0}^{j-1}(1-g(\mu_i))\right)
\left(\frac{-\frac{\partial }{\partial  \mu_0} g(\mu_0)}{(1-g(\mu_0))},\,\frac{-\frac{\partial }{\partial  \mu_1} g(\mu_1)}{(1-g(\mu_1))},\,\cdots,\,\frac{-\frac{\partial }{\partial  \mu_{j-1}} g(\mu_{j-1})}{(1-g(\mu_{j-1}))}\right),
\end{align*}
which is also regarded as a column vector in $\bigoplus\limits_{i=0}^{j-1}\mathbb{C}^{MN+1}$.

For every $\epsilon>0$, choose $\delta=\epsilon$, for every $w_1=({w_1}_i),\,w_2=({w_2}_i)\in A$ with $|w_1-w_2|_{l^{\infty}(\mathbb{C}^{N+1})}<\delta$,  by (\ref{product-derivative}) and the Integral Mean Value Theorem,  we have
\begin{align*}
 |G(w_1)-G(w_2)|_{l^{\infty}(\mathbb{C}^{N+1})} &=\sup_{j\in\mathbb{N}}\frac{1}{c^j}\left|\prod_{i=0}^{j-1}(1-g(\chi_i(w_1)  ))-\prod_{i=0}^{j-1}(1-g(\chi_i(w_2)))\right|\\
&\leq  \sup_{j\in\mathbb{N}}\frac{1}{c^j}\left|\left(\prod_{i=0}^{j-1}(1-g(\bar{\chi}_i  ))\right)
 \sum_{i=0}^{j-1}\frac{-\frac{\partial }{\partial  \chi_i} g(\bar{\chi}_i  )}{(1-g(\bar{\chi}_i  ))}\left(\chi_i(w_1)-\chi_i(w_2)\right)\right|,
\end{align*}
where  $\bar{\chi}_i=\chi_i(w_1)+\theta (\chi_i(w_1)-\chi_i(w_2))$ for some $\theta\in [0,\,1]$. By (A2) we have
$l<|1-g(\bar{\chi}_i  )|<c$. By (A1), there exists $M_0>0$ so that   $ |\frac{\partial }{\partial  \chi_i} g(\bar{\chi}_i  )|<M_0$.  By Lemma~\ref{Lemma-l-infty}, there exists $M_1>0$ so that   $\sup_{j\in\mathbb{N}}\frac{j}{c^j} \prod_{i=0}^{j-1}|1-g(\bar{\chi}_i  )| <M_1$.  It follows that
\begin{align}\label{G-continuity}
 |G(w_1)-G(w_2)|_{l^{\infty}(\mathbb{C}^{N+1})} &\leq \sup_{j\in\mathbb{N}}\frac{1}{c^j}\left(\prod_{i=0}^{j-1}|1-g(\bar{\chi}_i  )|\right)
\sum_{i=0}^{j-1}\frac{M_0}{l}\left| (\chi_i(w_1)-\chi_i(w_2))\right|\notag\\
&\leq \sup_{j\in\mathbb{N}}\frac{1}{c^j}\left(\prod_{i=0}^{j-1}|1-g(\bar{\chi}_i  )|\right)
\frac{jM_0}{l}|w_1-w_2|_{l^{\infty}(\mathbb{C}^{N+1})}\notag\\
&= \sup_{j\in\mathbb{N}}\frac{j}{c^j}\left(\prod_{i=0}^{j-1}|1-g(\bar{\chi}_i  )|\right)
\frac{M_0}{l}|w_1-w_2|_{l^{\infty}(\mathbb{C}^{N+1})}\notag\\
&= \frac{M_0M_1}{l}  \epsilon,
\end{align}which implies that $G$ is continuous.
Next, we show that for every $w=(w_i)\in A\subset l^{\infty}(U\times V)$, and for all $h=(h_i)\in l^{\infty}(\mathbb{C}^{N+1})$, the mapping $\mathscr{G}:t\rightarrow G(w+th)$ is analytic in the neighborhood of $0\in \mathbb{C}$. Denote by $\bar{G}h$ the sequence
\[ \left(\frac{1}{c^j} \left(\prod_{i=0}^{j-1}(1-g(\chi_i(w)))\right)
\sum_{i=0}^{j-1}\frac{-\frac{\partial }{\partial  \chi_i} g(\chi_i(w))}{(1-g(\chi_i(w) ))}\chi_i(h)\right)_{j=1}^{\infty}.\]  Then by the same argument leading to (\ref{G-continuity}), we know that $\bar{G}h\in  l^{\infty}(\mathbb{C}^{N+1})$ and
\begin{align}\label{analyticity-G}
 &\left|\frac{G(w+th)-G(w)}{t}-\bar{G}h\right|_{l^{\infty}(\mathbb{C}^{N+1})}\notag\\
= & \sup_{j\in\mathbb{N}}\frac{1}{c^j}\left|\frac{1}{t}\left(\prod_{i=0}^{j-1}(1-g(\chi_i(w+th)  ))-\prod_{i=0}^{j-1}(1-g(\chi_i(w)))\right)\right.\notag\\
 & \left.-\left(\prod_{i=0}^{j-1}(1-g(\chi_i(w)))\right)
\sum_{i=0}^{j-1}\frac{-\frac{\partial }{\partial  \chi_i} g(\chi_i(w) )\chi_i(h)}{ 1-g(\chi_i(w) ) }\right|\notag\\
= & \sup_{j\in\mathbb{N}}\frac{1}{c^j}\left|\left(\prod_{i=0}^{j-1}(1-g(\tilde{\chi}_i  ))\right)
\sum_{i=0}^{j-1}\frac{-\frac{\partial }{\partial  \chi_i} g(\tilde{\chi}_i  )\chi_i(h)}{1-g(\tilde{\chi}_i )}\right.\notag\\
& \left.-\left(\prod_{i=0}^{j-1}(1-g(\chi_i(w)  ))\right)
\sum_{i=0}^{j-1}\frac{-\frac{\partial }{\partial \chi_i} g(\chi_i(w))\chi_i(h)}{1-g(\chi_i(w)  )}\right|\notag\\
\leq  & \sup_{j\in\mathbb{N}}\frac{1}{c^j}\left|\left(\prod_{i=0}^{j-1}(1-g(\tilde{\chi}_i  ))-\prod_{i=0}^{j-1}(1-g({\chi}_i(w)))\right)
\sum_{i=0}^{j-1}\frac{-\frac{\partial }{\partial  \chi_i} g(\tilde{\chi}_i  )\chi_i(h)}{1-g(\tilde{\chi}_i)}\right|\notag\\
 & +\sup_{j\in\mathbb{N}}\frac{1}{c^j}\left|\left(\prod_{i=0}^{j-1}(1-g(\chi_i(w)))\right)
\left(\sum_{i=0}^{j-1}\frac{-\frac{\partial }{\partial\chi_i} g(\tilde{\chi}_i)\chi_i(h)}{1-g(\tilde{\chi}_i)}- 
\sum_{i=0}^{j-1}\frac{-\frac{\partial }{\partial  u_i} g(\chi_i(w))\chi_i(h)}{1-g(\chi_i(w))}\right)\right|, 
\end{align}
where $\tilde{\chi}_i=\chi_i(w+t\theta\, h)$ for some $\theta\in [0,\,1]$. By applying the same argument leading to (\ref{G-continuity}) on the first term of the last inequality of (\ref {analyticity-G}) and   by Lemma~\ref{Lemma-l-infty} we have
\begin{align}\label{Analyticity-G-1}
& \lim_{t\rightarrow 0}\sup_{j\in\mathbb{N}}\frac{1}{c^j}\left|\left(\prod_{i=0}^{j-1}(1-g(\tilde{\chi}_i  ))-\prod_{i=0}^{j-1}(1-g(\chi_i(w)  ))\right)
\sum_{i=0}^{j-1}\frac{-\frac{\partial }{\partial \chi_i} g(\tilde{\chi}_i  )\chi_i(h)}{1-g(\chi_i(w)  )}\right|\notag\\
   \leq & \lim_{t\rightarrow 0}\sup_{j\in\mathbb{N}}\frac{1}{c^j}\left|\left(\prod_{i=0}^{j-1}(1-g(\tilde{\chi}_i  ))-\prod_{i=0}^{j-1}(1-g(\chi_i(w)))\right)
\right|\frac{jM_0}{l}|h|_{l^{\infty}(\mathbb{C}^{N+1})}\notag\\
 \leq & \lim_{t\rightarrow 0}\sup_{j\in\mathbb{N}}\frac{1}{c^j}\left|\left(\prod_{i=0}^{j-1}(1-g(\tilde{\tilde{\chi}}_i  ))\right)
\sum_{i=0}^{j-1}\frac{-\frac{\partial }{\partial \chi_i} g(\tilde{\tilde{\chi}}_i  )}{(1-g(\tilde{\tilde{\chi}}_i  ))}\theta \, t \chi_i(h)\right|\frac{jM_0}{l}|h|_{l^{\infty}(\mathbb{C}^{N+1})}\notag\\
 = & \lim_{t\rightarrow 0}\sup_{j\in\mathbb{N}}\frac{1}{c^j}\left| \prod_{i=0}^{j-1}(1-g(\tilde{\tilde{\chi}}_i  )) 
\right|\frac{j^2M_0^2}{l^2}|h|_{l^{\infty}(\mathbb{C}^{N+1})}^2\cdot |t|\notag\\
= & \lim_{t\rightarrow 0}\sup_{j\in\mathbb{N}}\frac{j^2}{c^j}\left| \prod_{i=0}^{j-1}(1-g(\tilde{\tilde{\chi}}_i  )) 
\right|\frac{M_0^2}{l^2}|h|_{l^{\infty}(\mathbb{C}^{N+1})}^2\cdot |t|\notag\\
= & \,0.
\end{align}
where $\tilde{\tilde{\chi}}_i=\chi_i(w+ t\theta\theta'h)$ for some $\theta'\in [0,\,1].$  By (A1), there exists $M_2>0$ so that     $ |\frac{\partial^2 }{\partial  \chi_i^2} g({\chi}_i(w)  )|<M_2$ for every  $w\in A$ and  $i\in\mathbb{N}$.
Then it follows from the Integral Mean Value Theorem that the second term of the last inequality of (\ref {analyticity-G}) satisfies that
\begin{align*}
& \sup_{j\in\mathbb{N}}\frac{1}{c^j}\left|\left(\prod_{i=0}^{j-1}(1-g(\tilde{\chi}_i))\right)
\left(\sum_{i=0}^{j-1}\frac{-\frac{\partial }{\partial \chi_i} g(\tilde{\chi}_i  )\chi_i(h)}{(1-g(\tilde{\chi}_i  ))}- 
\sum_{i=0}^{j-1}\frac{-\frac{\partial }{\partial \chi_i} g(\chi_i(w))\chi_i(h)}{(1-g(\chi_i(w) ))}\right)\right| \notag\\
= &  \sup_{j\in\mathbb{N}}\frac{1}{c^j}\left|\left(\prod_{i=0}^{j-1}(1-g(\tilde{\chi}_i  ))\right)\left[
\left(\sum_{i=0}^{j-1}\frac{-\frac{\partial }{\partial  \chi_i} g(\tilde{\chi}_i  )\chi_i(h)}{(1-g(\tilde{\chi}_i  ))}-  
\sum_{i=0}^{j-1}\frac{-\frac{\partial }{\partial  \chi_i} g(\chi_i(w))\chi_i(h)}{(1-g(\tilde{\chi}_i  ))} \right)\right.\right. \notag\\
&\left.\left. +\sum_{i=0}^{j-1}\frac{-\frac{\partial }{\partial  \chi_i} g(\chi_i(w))\chi_i(w)}{(1-g(\tilde{\chi}_i  ))}- 
\sum_{i=0}^{j-1}\frac{-\frac{\partial }{\partial \chi_i} g(\chi_i(w))\chi_i(w)}{(1-g(\chi_i(w)  ))}\right] \right|\notag\\
\leq &  \sup_{j\in\mathbb{N}}\frac{1}{c^j}\left|\left(\prod_{i=0}^{j-1}(1-g(\tilde{\chi}_i  ))\right)\left[
\frac{jM_2t}{l}|h|_{l^{\infty}(\mathbb{C}^{N+1})}^2\right.\right. \notag\\
&\left.\left. +\left(-\frac{\partial }{\partial  \chi_i} g(\chi_i(w)  )\chi_i(h)\right)\sum_{i=0}^{j-1}\left(\frac{1}{(1-g(\tilde{\chi}_i  ))}- 
 \frac{1}{(1-g(\chi_i(w)))}\right)\right] \right|\notag\\
= &  \sup_{j\in\mathbb{N}}\frac{1}{c^j}\left|\left(\prod_{i=0}^{j-1}(1-g(\tilde{\chi}_i  ))\right)\left[
\frac{jM_2t}{l}|h|_{l^{\infty}(\mathbb{C}^{N+1})}^2\right.\right. \notag\\
&\left.\left. +M_0|h|_{l^{\infty}(\mathbb{C}^{N+1})}\sum_{i=0}^{j-1}\left(\frac{g(\tilde{\chi}_i  )-g(\chi_i(w))}{(1-g(\tilde{\chi}_i  ))(1-g(\chi_i(w)))} 
\right)\right] \right|\notag\\
\leq  &  \sup_{j\in\mathbb{N}}\frac{1}{c^j}\left|\left(\prod_{i=0}^{j-1}(1-g(\tilde{\chi}_i  ))\right)\left[
\frac{jM_2 |h|_{l^{\infty}(\mathbb{C}^{N+1})}^2t}{l} +\frac{jM_0^2|h|_{l^{\infty}(\mathbb{C}^{N+1})}^2t}{l^2} \right] \right|.
\end{align*}
Then by Lemma~\ref{Lemma-l-infty} we have
\begin{align}\label{Analyticity-G-2}
 \lim_{t\rightarrow 0}\sup_{j\in\mathbb{N}}\frac{1}{c^j}\left|\left(\prod_{i=0}^{j-1}(1-g(\tilde{\chi}_i  ))\right)
\left(\sum_{i=0}^{j-1}\frac{-\frac{\partial }{\partial  \chi_i} g(\tilde{\chi}_i  )\chi_i(h)}{(1-g(\tilde{\chi}_i  ))}- 
\sum_{i=0}^{j-1}\frac{-\frac{\partial }{\partial  \chi_i} g(\chi_i(w))\chi_i(h)}{(1-g(\chi_i(w)  ))}\right)\right|=0.
\end{align}
By (\ref{analyticity-G}), (\ref{Analyticity-G-1}) and (\ref{Analyticity-G-2}) we have
\[
 \lim_{t\rightarrow 0}\left|\frac{G(w+th)-G(w)}{t}-\bar{G}h\right|_{l^{\infty}(\mathbb{C}^{N+1})}=0.
\]
\hfill\hspace*{1em}\qed
\end{proof}

\begin{lemma}\label{complete-continuity}Assume $($\textrm{A1 -- A2}$\,)$. 
Let the set $A$ and the map $G$ be as in Lemma~\ref{Lemma-G}. 
 Define $H: A\subset l^{\infty} (U\times V)\rightarrow l^{\infty} (\mathbb{K}^{N+1})$ by 
 \[
 H(\theta)=(F_j(\theta)G_j(\theta))_{j=1}^{\infty}\in l^{\infty} (\mathbb{K}^{N+1}),
 \]
where
\begin{align*}
\theta = & (\theta_1,\,\theta_2,\,\cdots,\,\theta_j,\,\cdots)=((u_1,\,v_1),\,(u_2,\,v_2),\,\cdots,\,(u_j,\,v_j),\,\cdots)\in l^{\infty} (\mathbb{K}^{N+1}),\\
F_j(\theta)= & \,\left(\frac{f(u_j,\, u_{j+1}  )}{1-g(u_j,\,u_{j+1},\,\cdots,u_{j+M-1}, \,v_j  )}, \frac{g(u_j,\,u_{j+1},\,\cdots,u_{j+M-1}, \,v_j  )}{1-g(u_j,\,u_{j+1},\,\cdots,u_{j+M-1}, \,v_j  )}\right),
\end{align*} for $j\geq 1,\,j\in\mathbb{N}$.
Then $H: \bar{A}\rightarrow l^{\infty} (\mathbb{K}^{N+1})$ is completely continuous and is analytic. 
\end{lemma}
\begin{proof} According to Lemma~\ref{l-m-space}, we only need to show that for every bounded set $B\subset  l^{\infty} (\mathbb{K}^{N+1})$, $H(B)$ is bounded in $ l_m^{\infty} (\mathbb{K}^{N+1})$. By (A1)-(A2), we know that $F(B)$ is bounded in  $ l^{\infty} (\mathbb{K}^{N+1})$. Then by Lemma~\ref{Lemma-l-infty}, we have
\begin{align*}
\lim_{j\rightarrow+\infty}{j^m} \left|F_j(\theta)G_j(\theta)\right|=0.
\end{align*}Therefore we have $H(\theta)\in l_m^{\infty} (\mathbb{K}^{N+1})$. By  Lemma~\ref{l-m-space}, $H$ is completely continuous. Then by (A1)--(A2), $F:  l^{\infty} (\mathbb{C}^{N+1})\ni \theta\rightarrow F(\theta)\in l^{\infty} (\mathbb{C}^{N+1})$ is analytic. Then by the similar procedure for the proof of Lemma~\ref{Lemma-G}, we can  show that for every $w=(w_i)\in A\subset l^{\infty}(U\times V)$, and for all $h=(h_i)\in l^{\infty}(\mathbb{C}^{N+1})$, the mapping $\mathscr{G}:t\rightarrow H(w+th)$ is analytic in the neighborhood of $0\in \mathbb{C}$.
\hfill\hspace*{1em}\qed
\end{proof}

\section{Main Results}\label{Main-results}
 
Let $\Omega$ be a bounded closed ball in $\mathbb{K}$. We denote by  $C(\Omega; l^{\infty}(\mathbb{K}^{N+1}))$   the space of  continuous   functions $u: \Omega\ni t\rightarrow u(t)\in l^{\infty}(\mathbb{K}^{N+1})$ and denote by  $C^1(\Omega; l^{\infty}(\mathbb{K}^{N+1}))$ the space of  continuously  differentiable   functions $u: \Omega\ni t\rightarrow u(t)\in l^{\infty}(\mathbb{K}^{N+1})$. Then it is clear that $C(\Omega; l^{\infty}(\mathbb{K}^{N+1}))$ 
  and $C^1(\Omega; l^{\infty}(\mathbb{K}^{N+1}))$ are Banach spaces equipped, respectively, with the norms
$\|u\|   =\max_{t\in  \Omega }|u(t)|_{l^{\infty}(\mathbb{K}^{N+1}))}$   and 
\[\|u\|  =\max\{\max_{t\in  \Omega }|u(t)|_{l^{\infty}(\mathbb{K}^{N+1})},\,\max_{t\in  \Omega }|u'(t)|_{ {l}^{\infty}(\mathbb{K}^{N+1})}\}. 
\]

 \begin{theorem}\label{Analyticity-th}Assume $($\textrm{A1 -- A2}$\,)$.    Let $(x,\,\tau)\in\mathbb{R}^{N+1}$ be a bounded solution  of  system (\ref{eqn-2}). Suppose that there exists a compact set $Q\subset U\times V$ such that $(x(t),\,\tau(t))\in Q$ for all $t\in\mathbb{R}$. Then  $(x,\,\tau)$ is analytic on $\mathbb{R}$. 
 \end{theorem}
 \begin{proof}   We define  $\left((y_j,\,z_j)\right)_{j=1}^\infty\in C(\mathbb{R}; l_c^{\infty}(\mathbb{R}^{N+1}))$ by 
\begin{align*}(y_j(t),\, z_j(t))=\left(\frac{1}{c^{j}}x(\eta^{j-1}(t)),\,\frac{1}{c^{j}}\tau(\eta^{j-1}(t))\right)
\mbox{ for $j\geq 1,\,j\in\mathbb{N},\,t\in\mathbb{R}$.} 
\end{align*}
Then by the derivation in Section~\ref{SOPS-4-1},  for every $t\in\mathbb{R}$, $((y_j(t),\,z_j(t)))_{j=1}^{\infty}\in l_c^{\infty}(\mathbb{R}^{N+1})$ satisfies system~(\ref{New-Eqn-1}).

 Let  
\begin{align}\label{map-F}
 F(\theta)=(F_1(\theta),\,F_2(\theta),\,\cdots,\,F_j(\theta),\,\cdots),
\end{align}
where
\begin{align*}
\theta = & (\theta_1,\,\theta_2,\,\cdots,\,\theta_j,\,\cdots)=((u_1,\,v_1),\,(u_2,\,v_2),\,\cdots,\,(u_j,\,v_j),\,\cdots)\in l^{\infty} (\mathbb{R}^{N+1}),\\
F_j(\theta)= & \,\left(\frac{f(u_j,\, u_{j+1}  )}{1-g(u_j,\,u_{j+1},\,\cdots,u_{j+M-1}, \,v_j  )}, \frac{g(u_j,\,u_{j+1},\,\cdots,u_{j+M-1}, \,v_j  )}{1-g(u_j,\,u_{j+1},\,\cdots,u_{j+M-1}, \,v_j  )}\right),
\end{align*} for $j\geq 1,\,j\in\mathbb{N}$.

Let  $T$ be as in Lemma~\ref{OMT}, $G$ in Lemma~\ref{Lemma-G}. Then $w=\left((y_1,\,z_1),\, (y_2,\,z_2),\,\cdots)\right)\in C(\mathbb{R};l_c^{\infty} (\mathbb{R}^{N+1}))$ is a solution of the following ordinary differential equation
\begin{align}\label{Abs-ODE-S-infty}
 \frac{d}{dt}w(t)= H(Tw(t)), 
\end{align}
where $H(T(w))=(F_1(T(w))G_1(T(w)),\,F_2(T(w))G_2(T(w)),\,\cdots)\in l^{\infty} (\mathbb{R}^{N+1})$ and  $G_j$ is the $j$-th coordinate of $G$. 
Moreover, we notice that for every $j\in\mathbb{N}$,
\[
\{(c^jy_j(t),\,c^jz_j(t)): t\in\mathbb{R}\}=\{(x(t),\,y(t)):t\in\mathbb{R}\}\subset Q.
\] Then we have
$\{(c^jy_j(t),\,c^jz_j(t))_{j=1}^\infty: t\in\mathbb{R}\}\subset Q$ and 
\[
Tw(t)=(c^jy_j(t),\,c^jz_j(t))_{j=1}^\infty\in A
\]for every $t\in\mathbb{R}$, where    $A$ is defined by (\ref{set-A}).
Let   $w_{t_0}=\left((y_j(t_0),\,z_j(t_0))\right)_{j=1}^\infty\in l_c^{\infty} (\mathbb{R}^{N+1})$, $t_0\in\mathbb{R}$. Then  $w(t)$ is a solution of the following initial value problem 
\begin{align}\label{Abs-ODE-IVP}\left\{
\begin{aligned}
 \frac{d}{dt}w(t)& = H(Tw(t)), \\
w(t_0)&= w_{t_0}.
\end{aligned}
\right.
\end{align}
To prove the existence of complex extension of $w(t)\in l_c^{\infty} (\mathbb{R}^{N+1})$, 
  we put $\nu(t)=T w(t)\in  l^{\infty} (\mathbb{C}^{N+1})$ and consider equation (\ref{Abs-ODE-IVP}) in $l^{\infty} (\mathbb{C}^{N+1})$. Then equation (\ref{Abs-ODE-IVP})
is transformed into the following integral equation
\begin{align}\label{Abs-Integral-Eq} 
 T^{-1}{\nu}(t)= w_{t_0}+\int_{t_0}^tH(\nu(s)) ds,
\end{align}where the integral is taken along the linear path $\xi\rightarrow t_0+\xi(t-t_0)$, $0\leq \xi\leq 1$.

Denote by $\Omega_h=\{t\in\mathbb{C}: |t-t_0|\leq h\}$ for some $h>0$. To prove the  existence and uniqueness of the solution using the Uniform Contraction Principle, we consider fixed point problem associated with the following mapping
\begin{align}\label{L0-lambda-operator}
 L_0(\nu)(t)=(I-T^{-1})\nu(t) + w_{t_0}+\int_{t_0}^tH(\nu(s)) ds,  
\end{align} on $C(\Omega_h; A)$.    
However, by Lemma~\ref{Compact-operator}, $L_0$ is not contractive on $C(\Omega_h; l^\infty(\mathbb{C}^{N+1}))$ in general, but its perturbation 
 $L: C(\Omega_h; l^\infty(\mathbb{C}^{N+1}))\times [0,\,1] \rightarrow C(\Omega_h; l^\infty(\mathbb{C}^{N+1}))$ defined by
\begin{align}\label{L-lambda-operator}
 L(\nu,\,\lambda)(t)=((1-\lambda)I-T^{-1})\nu(t) + (T^{-1}+\lambda I)\nu_{t_0}+\int_{t_0}^tH(\nu(s)) ds, 
\end{align} is contractive on $C(\Omega_h; l^\infty(\mathbb{C}^{N+1}))$ for $\lambda\in (0,\,1-1/c)$ and some $h>0$, where $\nu_{t_0}=Tw_{t_0}$.

 If there exists a $\nu\in C(\Omega_h; A)$ such that $L(\nu)=\nu$, then $\nu$ is a solution of the following initial value problem:
\begin{align}\label{Abs-ODE-IVP-lambda}\left\{
\begin{aligned}
 (\lambda I+T^{-1})\frac{d}{dt}\nu(t)& = H(\nu(t)), \\
\nu(t_0)&=T w_{t_0}.
\end{aligned}
\right.
\end{align}  Writing (\ref{Abs-ODE-IVP-lambda}) in integral form, we have
 \begin{align}\label{Abs-ODE-IVP-integral-lambda}
\begin{aligned}
T^{-1}\nu  (t)& = w_{t_0}+\int_{t_0}^t(\lambda T+ I)^{-1} H\left(\nu (s)\right)ds.
\end{aligned}
\end{align}

By Lemma~\ref{complete-continuity}, $H$ is   analytic  and   $L$ is an analytic mapping from    $A\subset  l^{\infty} (\mathbb{C}^{N+1})$ to $ l^{\infty} (\mathbb{C}^{N+1})$. 

We organize the remaining part of the proof as follows: We first show with claims 1 and 2 the existence and uniqueness of  solutions $\nu_\lambda$ of   system~(\ref{Abs-ODE-IVP-integral-lambda})  and with claim 3 $\nu_\lambda$ satisfies that $\lim_{\lambda\rightarrow 0^+}T^{-1}\nu_\lambda =w_0\in  C(\Omega_{h_0}; T^{-1}(\bar{A}))$ for some $h_0>0$. Secondly, we show with claim 4 that the right hand side of  system~(\ref{Abs-ODE-IVP-integral-lambda}) is coordinate-wise convergent to that of system~(\ref{Abs-Integral-Eq}) with 
 $\lim_{n\rightarrow+\infty}H(\nu_{\lambda_n}(t))=H(\nu_0(t))$, where  $\nu_0: \Omega_{h_0}\rightarrow\bar{A}$ is a map such that $H(\nu_0)$ is continuous in $t\in \Omega_{h_0}$ and the sequence $\{\lambda_n\}_{n=1}^{\infty}$ is in $(0,\,1-\frac{1}{c})$ with $\lim_{n\rightarrow+\infty}\lambda_n=0$.  Lastly, we show with claim 5 that $H(Tw_0)=H(\nu_0)$ which implies that $w_0$ satisfies system~(\ref{Abs-Integral-Eq}) and hence it is the solution of the initial value problem (\ref{Abs-ODE-S-infty}).

Now we show the following

\textbf{Claim 1}: For every $\lambda\in (0,\,1)$, there exists $h>0$ such that  there exists 
one and only one point $\nu_\lambda\in   C(\Omega_h; \bar{A})$  such that $L(\nu_\lambda,\,\lambda)=\nu_\lambda$ and $\nu_\lambda$ is  analytic and is differentiable with respect to $\lambda$.

{\textit{Proof of  Claim 1:}}  We only need to show that $L$ is a contractive mapping in some closed neighborhood of $w_{t_0}$ in $C(\Omega_h; l^{\infty}(\mathbb{C}^{N+1}))$ where $\Omega_h=\{t\in\mathbb{C}: |t-t_0|\leq h\}$ for some $h>0$ to be determined. Denote by $\|\cdot\|_{C}$ the supremum norm on the Banach space ${C(\Omega_h; l^{\infty}(\mathbb{C}^{N+1}))}$.  For every $w_1,\,w_2\in C(\Omega_h; l^{\infty}(\mathbb{C}^{N+1}))$ we have
\begin{align}\label{eqn-condensing}
  &\|L(w_1,\,\lambda)-L(w_2,\,\lambda)\|_{C}\notag\\
=& \max_{t\in \Omega_h} \left\|((1-\lambda)I-T^{-1})w_1(t)+\int_{t_0}^tH(w_1(s)) ds\right)\notag\\
& \left.-((1-\lambda)I-T^{-1})w_2(t)-\int_{t_0}^tH(w_2(s)) ds \right\|_{l^{\infty}(\mathbb{C}^{N+1})}\notag\\
\leq &  \max_{t\in \Omega_h}  \|(1-\lambda) I -T^{-1}\|_{\mathscr{L}(l^{\infty}(\mathbb{K}^{N+1};\, l^{\infty}(\mathbb{K}^{N+1})} |w_1(t)-w_2(t)|_{l^{\infty}(\mathbb{C}^{N+1})}\notag \\
&  +\max_{t\in \Omega_h}\int_{t_0}^t \left\|H(w_1(s))-H(w_2(s))\right  \|_{l^{\infty}(\mathbb{C}^{N+1})} ds.
\end{align}
Since $H$ is analytic on $A$, there exist  constants $\delta>0$ and $l_0>0$ so that
$|H(\nu_1 )-H(\nu_2) |_{l^{\infty}(\mathbb{C}^{N+1})}\leq l_0 |\nu_1-\nu_2|_{l^{\infty}(\mathbb{C}^{N+1})}$ for every $\nu_1,\,\nu_2\in  A$ with $|\nu_1-\nu({t_0})|_{l^{\infty}(\mathbb{C}^{N+1})}\leq \delta, |\nu_2-\nu(t_0)|_{l^{\infty}(\mathbb{C}^{N+1})}\leq \delta$. 

Let $X=\{\nu\in C(\Omega_h;  \bar{A}): \max_{t\in\Omega_h}|\nu(t)-\nu({t_0})|_{l^{\infty}(\mathbb{C}^{N+1})}\leq \delta\}$. Then $X$ is a closed subset of the Banach space $C(\Omega_h; l^{\infty}(\mathbb{C}^{N+1})$.  By (\ref{eqn-condensing}) and by Lemma~\ref{Compact-operator}, we have
\begin{align*}
 \|L(w_1,\,\lambda)-L(w_2,\,\lambda)\|_{C}\leq & (1-\lambda)  \|w_1-w_2\|_{C}+ l_0 h \|w_1-w_2\|_{C}\\ = & (1-\lambda+l_0 h) \|w_1-w_2\|_{C},
\end{align*} for every $w_1,\,w_2\in X$.
Therefore, if $h\in (0,\,\frac{\lambda}{l_0})$, then $1-\lambda+l_0 h\in (0,\,1)$. Moreover,  we choose $h>0$
 small enough so that
\begin{align*}
 & \max_{t\in\Omega_h}\|L(\nu(t),\,\lambda)-\nu({t_0})\|_{l^{\infty}(\mathbb{C}^{N+1})}\\
=&\max_{t\in\Omega_h}\left\|((1-\lambda)I-T^{-1})(\nu(t)-\nu({t_0}))+\int_{t_0}^tH(\nu(s)) ds\right\|_{l^{\infty}(\mathbb{C}^{N+1})}\\
\leq & \,\delta. 
\end{align*}
Then by the Uniform Contraction Principle in Banach spaces, we know that $L(\cdot,\,\lambda): X\rightarrow X$ is a contractive mapping with a unique fixed point $\nu_\lambda\in   C(\Omega_h;  \bar{A})$ and $\nu_{\lambda}$ is analytic. Noticing that $L$ is linear in $\lambda$,  $\nu_{\lambda}$ is differentiable with respect to $\lambda$. This completes the proof of Claim 1.

\textbf{Claim 2}: There exists $h_0>0$ and   so that $\Omega_{h_0}$ is the common existence region of the fixed points $\nu_\lambda$  of $L(\nu,\,\lambda)$ for all $\lambda\in (0,\,1-1/c).$

{\textit{Proof of  Claim 2:}} 
Let $w_\lambda= T^{-1}\nu_\lambda,\,\nu_\lambda\in X$ where   $X$ is as in Claim 1.  
Note that $\nu_{\lambda}\in C(\Omega_{h_\lambda};   \bar{A})$ where $h_\lambda>0$ is a constant depending on $\lambda$. 
Let $\widetilde{M}>0$ be the supremum of $\|H(\nu)\|_{l^{\infty}(\mathbb{C}^{N+1})}$ on  $\bar{A}$. Let $0< \beta\leq +\infty$ be such that $\{t\in\mathbb{C}: |t-t_0|<\beta\}$ is the maximal existence region of $\nu_\lambda(t)$ on $\bar{A}$. If $\beta=+\infty$, then $\nu_\lambda$ can be extended to the whole complex plane $\mathbb{C}$ with $\nu_\lambda(t)\in l^{\infty}(\bar{U}\times \bar{V}))$ for all $t\in\mathbb{C}$. Otherwise, by Theorem 10.5.5 of \cite{Dieudonne}, there exists $t_1 \in\{t\in\mathbb{C}: |t-t_0|<\beta\}$  so that
$\nu_\lambda$ achieves value  in the boundary of  $A$. Let $B$  denote
the boundary of  $A$. Let $r$ be defined by
\begin{align*}
r=  \inf_{\nu\in B}\|T^{-1}(\nu-\nu_{t_0})\|_{l^{\infty}(\mathbb{C}^{N+1})}.
\end{align*}
Now we show that $r>0$. Suppose not.  Note that by Lemma~\ref{OMT},  $T^{-1}$ is compact and $B$ is closed and bounded in $l^{\infty}(\mathbb{C}^{N+1})$. Therefore $r$ is the minimum norm of a compact set. There exists $\nu^*\in B$ such that $r= \|T^{-1}(\nu^*-\nu_{t_0})\|_{l^{\infty}(\mathbb{C}^{N+1})}=0.$
 Then 
 we have $\nu_{t_0}=\nu(t_0)=\nu^*\in B$. This is a contradiction since $\nu(t_0)$ is in the interior of $A$. It follows that $r>0$.

By Lemma~\ref{Compact-operator}, we know that  $\lambda I +T^{-1}\in \mathscr{L}(l^{\infty}(\mathbb{C}^{N+1}); l^{\infty}(\mathbb{C}^{N+1}))$ has norm equal to $ \lambda+\frac{1}{c}$. Then we have
\begin{align*}
r = & \inf_{\nu\in B}\|T^{-1}(\nu-\nu(t_0))\|_{l^{\infty}(\mathbb{C}^{N+1})}\\
\leq &\inf_{\nu\in B}\|(\lambda I +T^{-1})(\nu-\nu(t_0))\|_{l^{\infty}(\mathbb{C}^{N+1})}.\\
\leq & \left\|(\lambda I +T^{-1})(\nu_\lambda(t_1)-\nu(t_0))\right\|_{l^{\infty}(\mathbb{C}^{N+1})}\\
      = & \left\| \int_{t_0}^{t_1} (\lambda I +T^{-1}) \nu_\lambda' (s)ds\right\|_{l^{\infty}(\mathbb{C}^{N+1})}\\
    \leq & \sup_{t\in\Omega_h}\int_{t_0}^t \left\|  H(\nu_\lambda (s)) \right\|_{l^{\infty}(\mathbb{C}^{N+1})}ds\\
 \leq &\widetilde{M}\beta.
\end{align*} 
It follows that $\beta\geq \frac{r}{\widetilde{M}}$. Let $h_0=\frac{r}{\widetilde{M}}$.  Then $\Omega_{h_0}$ is the common existence region of  $\nu_\lambda$   for all $\lambda\in (0,\,1-1/c).$  This completes the proof of Claim 2.

\textbf{Claim 3}: Let $\nu_\lambda$, and  $h_0$  be as in Claim 2.  There exists an analytic function $w_0\in    C(\Omega_{h_0};T^{-1}(\bar{A}))$  so that   $\lim_{\lambda\rightarrow 0^+} \|T^{-1}\nu_{\lambda}-w_0\|_{C(\Omega_{h_0}; l^{\infty}(\mathbb{C}^{N+1}))}=0$.

{\textit{Proof of  Claim 3:}} By Claim 2, we have $w_\lambda=T^{-1}\nu_\lambda\in    C(\Omega_{h_0}; T^{-1}(\bar{A}))$.  Moreover,  the uniformly bounded set $\left\{w_{\lambda}: \lambda\in (0,\,1-1/c)\right\}$ is  compact  in $C(\Omega_{h_0}; T^{-1}(\bar{A}))$, by the Arzel\'{a}--Ascoli theorem, since for every $\varepsilon>0$ there exists $\tilde{\delta}=\frac{\varepsilon}{\widetilde{M}} >0$ so that $|t-t'|<\tilde{\delta}$ implies that
\begin{align*}
\left\|w_{\lambda}(t)-w_{\lambda}(t')\right\|_{l^{\infty}(\mathbb{C}^{N+1})}
\leq &  \left\|\int_{t'}^t (\lambda T+I)^{-1} H(Tw_{\lambda} (s))ds \right\|_{l^{\infty}(\mathbb{C}^{N+1})}\\
\leq & \widetilde{M} \tilde{\delta}\\
=&\varepsilon,
\end{align*}
where $\widetilde{M}>0$ was defined in the proof of Claim 2, and Lemma~\ref{Extension-operator} was applied to obtain the second inequality. 
Therefore, there exists $w_0\in C(\Omega_{h_0};  \bar{A}) $ so that
\begin{align}\label{w-uniform}\lim_{\lambda\rightarrow 0}\|w_{\lambda}-w_0\|_{C(\Omega_{h_0};  \bar{A})}=0.
\end{align} Since $\left\{w_{\lambda}\right\}_{\lambda\in (0,\,1-1/c)}$ is a set of analytic functions in norm $\|\cdot\|_{ C(\Omega_{h_0}; l_c^{\infty}(\mathbb{C}^{N+1}))}$ and analytic in  norm $\|\cdot\|_{C(\Omega_{h_0}; l^{\infty}(\mathbb{C}^{N+1}))}$,  $w_0$ is also analytic  in norm $\|\cdot\|_{C(\Omega_{h_0}; l^{\infty}(\mathbb{C}^{N+1}))}$.

Now we show that $w_0\in C(\Omega_{h_0}; T^{-1}(A))$. First we show that $w_0\in C(\Omega_{h_0}; l_c^{\infty}(\mathbb{C}^{N+1})$.
Suppose that $w_0\not\in C(\Omega_{h_0}; l_c^{\infty}(\mathbb{C}^{N+1})$. Then for every $K>0$ there exists $j_0\in\mathbb{N}$ such that $\sup_{t\in \Omega_{h_0}}c^{j_0}|(w_0)_{j_0}(t)|>K$. That is,
\begin{align}\label{MJ0}
 \sup_{t\in \Omega_{h_0}}|(w_0)_{j_0}(t)|>\frac{K}{c^{j_0}}.
\end{align}
On the other hand, it follows from
$\lim_{\lambda\rightarrow 0^+}\|w_{\lambda}-w_0\|_C=0$, that
for every $\epsilon>0$, there exists $\delta>0$ such that for every $\lambda\in (0,\,\delta)$, we have
\begin{align*}
 \sup_{t\in \Omega_{h_0}}\sup_{j\in\mathbb{N}}|(w_0)_{j}(t)-(w_\lambda)_{j}(t)|<\epsilon, 
\end{align*}
which leads to
\begin{align*}
 \sup_{t\in \Omega_{h_0}} |(w_0)_{j}(t)|-\sup_{t\in \Omega_{h_0}}|(w_\lambda)_{j}(t)|<\epsilon,\,\mbox{for every $j\in\mathbb{N}$}. 
\end{align*}
It follows that
\begin{align}\label{tri-inequality}
 \sup_{t\in \Omega_{h_0}}|(w_\lambda)_{j}(t)|>\sup_{t\in \Omega_{h_0}} |(w_0)_{j}(t)|-\epsilon,\,\mbox{for every $j\in\mathbb{N}$}. 
\end{align}
Choosing $j=j_0$ and $\epsilon=\frac{K}{2c^{j_0}}$ in (\ref{tri-inequality}), then by (\ref{MJ0})  we obtain that
\begin{align*}
 \sup_{t\in \Omega_{h_0}}|(w_\lambda)_{j_0}(t)|& \geq \sup_{t\in \Omega_{h_0}} |(w_0)_{j_0}(t)|-\frac{K}{2c^{j_0}}\\
 & >\frac{K}{2c^{j_0}},
\end{align*}which leads to
$ \sup_{t\in \Omega_{h_0}}|c^{j_0}(w_\lambda)_{j_0}(t)|>K/2$ for every $\lambda\in (0,\,\delta)$. That is, $w_\lambda\not\in C(\Omega_{h_0}; l_c^{\infty}(\mathbb{C}^{N+1}))$ as $\lambda\rightarrow 0$ and hence $\nu_\lambda=Tw_\lambda\not\in C(\Omega_{h_0}; l^{\infty}(\mathbb{C}^{N+1}))$. This is a contradiction and hence $w_0\in C(\Omega_{h_0}; l_c^{\infty}(\mathbb{C}^{N+1}))$.  

Next we show that $w_0\in C(\Omega_{h_0}; T^{-1}(\bar{A}))$. Suppose not. Since  $w_0\in C(\Omega_{h_0}; l_c^{\infty}(\mathbb{C}^{N+1}))$, there exists $t^*\in\Omega_{h_0}$ so that $w_0(t^*)\in l_c^{\infty}(\mathbb{C}^{N+1})\setminus  T^{-1}(\bar{A})$.  By (\ref{w-uniform}) we have
\begin{align}\label{w-t-star}
\lim_{\lambda\rightarrow 0}\|w_\lambda(t^*)-w_0(t^*)\|_{l^{\infty}(\mathbb{C}^{N+1})}=0.
\end{align}
Since $\left\{w_{\lambda}: \lambda\in (0,\,1-1/c)\right\}$ is  uniformly bounded  in $C(\Omega_{h_0}; T^{-1}(\bar{A}))$, there exists a closed ball 
 $B'$ in $T^{-1}(\bar{A})$ which contains the  closure of $\{w_\lambda(t^*)\}_{\lambda\in (0,\,1-1/c)}$.  Then by Lemma~\ref{Banach-spaces}
and by (\ref{w-t-star}), we have $w_0(t^*)\in B'\subset T^{-1}(\bar{A})$ which is a contradiction. This completes the proof of Claim 3.

{\bf Claim 4:} Let  $h_0$ be as in Claim 2.  There exists a map $\nu_0: \Omega_{h_0}\rightarrow \bar{A}$ such that $H(\nu_0)$ is continuous and is such that for every $t\in \Omega_{h_0}$,
 there exists a sequence $\{\lambda_n\}_{n=1}^{\infty}\subset  (0,\,1-\frac{1}{c})$ with $\lim_{n\rightarrow+\infty}\lambda_n=0$  such that
 $\lim_{n\rightarrow+\infty}H(\nu_{\lambda_n}(t))=H(\nu_0(t)).$

 {\it Proof of Claim 4}: Note that by Claim 1, 
$\nu_\lambda \in C(\Omega_{h_0};T^{-1}(\bar{A}))$ is uniformly bounded with respect to $\lambda\in (0,\,1-\frac{1}{c})$. 
Since by Lemma~\ref{complete-continuity} $H$ is completely continuous, for every $t\in \Omega_{h_0}$, the set \[
\left\{H(\nu_\lambda(t)): \lambda\in \left(0,\,1-\frac{1}{c}\right)\right\},
\]is pre-compact in $l^\infty(\mathbb{C}^{N+1})$.
 So there exists a sequence $\{\lambda_n\}_{n=1}^{\infty}\subset  (0,\,1-\frac{1}{c})$ with $\lim_{n\rightarrow+\infty}\lambda_n=0$ and $\nu_0(t)\in T^{-1}(\bar{A})$, where $T^{-1}(\bar{A})$ is compact,  such that
 \begin{align}\label{new-limit-01}  
 \lim_{n\rightarrow+\infty}H(\nu_{\lambda_n}(t))=\lim_{n\rightarrow+\infty} H(Tw_{\lambda_n}(t))=H(\nu_0(t)).
\end{align}
 Next we show that $H(\nu_0): \Omega_{h_0}\ni t\rightarrow H(\nu_0(t)) \in l^\infty(\mathbb{C}^{N+1})$ is continuous in $t\in \Omega_{h_0}$. Let $t\in \Omega_{h_0}$. By (\ref{new-limit-01}), for every $\epsilon>0$, there exists $N_1\in\mathbb{N}$ such that for every $n>N_1$, 
 \begin{align}\label{nu-1}
 \|H(\nu_{\lambda_n}(t))-H(\nu_0(t))\|_{ l^\infty(\mathbb{C}^{N+1})}<\frac{\epsilon}{3},
 \end{align}
 Since $H(\nu_{\lambda_n})$ is continuous, there exists $\delta>0$ such that for every $t'\in  \Omega_{h_0} $ with $|t-t'|<\delta$ we have
  \begin{align}\label{nu-2}
 \|H(\nu_{\lambda_n}(t))-H(\nu_{\lambda_n}(t'))\|_{ l^\infty(\mathbb{C}^{N+1})}<\frac{\epsilon}{3}.
 \end{align}
Taking subsequence of $\{\lambda_n\}$ if necessary, by (\ref{new-limit-01}) there exists $N'$ such that for every $n>N'$,  we have
 \begin{align}\label{nu-3}
 \|H(\nu_{\lambda_n}(t'))-H(\nu_0(t'))\|_{ l^\infty(\mathbb{C}^{N+1})}<\frac{\epsilon}{3}.
 \end{align}
 By (\ref{nu-1}), (\ref{nu-2}) and (\ref{nu-3}) we have for $n>\max\{N_1,\,N'\}$,
 \[
 \|H(\nu_0(t))-H(\nu_0(t'))\|_{ l^\infty(\mathbb{C}^{N+1})}<\epsilon.
 \]That is $H(\nu_0)$ is continuous. This completes the proof of Claim 4.

 {\bf Claim 5:} Let  $h_0$ be as in Claim 2, $w_0$ be as in Claim 3,  $\nu_0$  be as in Claim 4. Then $H(Tw_0)=H(\nu_0)$ and $w_0$ is the solution of  the initial value problem (\ref{Abs-ODE-IVP}).

 {\it Proof of Claim 5}:
 It follows from  Claim 3 that $w_0$ is in $C(\Omega_{h_0}; T^{-1}(\bar{A}))$ and  $w_0$ is the limit of $w_\lambda$ as $\lambda\rightarrow 0^+$ in the norm $\|\cdot\|_{C(\Omega_{h_0}; l^\infty(\mathbb{C}^{N+1}))}$.   We first show that $Tw_{\lambda} $ converges to $Tw_0$ coordinate-wise. That is, for every $j\in\mathbb{N}$, 
 \begin{align}\label{claim-5-1}
 \lim_{\lambda\rightarrow 0}\sup_{t\in \Omega_{h_0}}|(Tw_{\lambda})_j(t)-(Tw_0)_j(t)|= 0.
 \end{align}
If not, there exists $j_0\in\mathbb{N}$ and $\epsilon_0>0$ and a sequence  $\{\lambda_{n}\}_{n=1}^{\infty}\subset (0,\,1-\frac{1}{c}) $ converging to 0 such that
\[
\sup_{t\in \Omega_{h_0}}|(w_{\lambda_{n}})_{j_0}(t)-(w_0)_{j_0}(t)|\geq \frac{\epsilon_0}{c^{j_0}}, \mbox{for all } n\in\mathbb{N},
\]which leads to
\[
\sup_{t\in \Omega_{h_0}}\sup_{j\in\mathbb{N}}|(w_{\lambda_{n}})_{j}(t)-(w_0)_{j}(t)|\geq\sup_{t\in \Omega_{h_0}}|(w_{\lambda_{n}})_{j_0}(t)-(w_0)_{j_0}(t)|\geq \frac{\epsilon_0}{c^{j_0}}, 
\]
 for all $n\in\mathbb{N}$.  This is a contradiction, since $w_0$ is the limit of $w_\lambda$ as $\lambda\rightarrow 0^+$ in the norm $\|\cdot\|_{C(\Omega_{h_0}; l^\infty(\mathbb{C}^{N+1}))}$.

Noticing that each coordinate of $H(Tw_\lambda)$ involves only finitely many coordinates of  $Tw_\lambda$ and $H$ is analytic.  $H(Tw_{\lambda})$ converges to $H(Tw_0)$ coordinate-wise as $Tw_{\lambda} $ converges to $Tw_0$ coordinate-wise with $\lambda\rightarrow 0^+$. By Claim 4, we have
\begin{align}\label{eqn-last0}
H(\nu_0)=H(Tw_0).
\end{align}
Noting that by  Lemma~\ref{Extension-operator},  $(\lambda T+I)^{-1}\in \mathscr{L}(l^{\infty}(\mathbb{C}^{N+1}); l^{\infty}(\mathbb{C}^{N+1}))$ is bounded for every $\lambda\in[0,\,1)$. Notice that $\nu_\lambda$, $\lambda\in (0,\,1-\frac{1}{c})$, satisfies (\ref{Abs-ODE-IVP-integral-lambda}). On the one hand, by Claim 3 we have
\begin{align}\label{eqn-last2}
\lim_{\lambda\rightarrow 0^+}\|T^{-1}\nu_{\lambda} (t)-w_0(t)\|_{l^\infty(\mathbb{C}^{N+1})}=0.
\end{align}On the other hand, for every $j\in\mathbb{N}$ and $t\in\Omega_{h_0}$ we have 
\begin{align}\label{eqn-last}
& \left|\int_{t_0}^t\left[(\lambda T+ I)^{-1} H\right]_j\left(\nu_{\lambda}(s)\right)ds-\int_{t_0}^t H_j(\nu_0(s))ds\right|\notag\\
 = & \left|\int_{t_0}^t\left[(\lambda T+ I)^{-1} H\right]_j\left(\nu_{\lambda}(s)\right)- H_j(\nu_0(s))ds\right|\notag\\
= &\left|\int_{t_0}^t  \frac{1}{(\lambda c^j+ 1)} H_j\left(\nu_{\lambda}(s)\right)- H_j(\nu_0(s)) ds\right|\notag\\
= &\left|\int_{t_0}^t \left( \frac{1}{(\lambda c^j+ 1)} (H_j\left(\nu_{\lambda}(s)\right)- H_j(\nu_0(s)))-\frac{\lambda c^j}{\lambda c^j+1} H_j(\nu_0(s))\right)ds\right|\notag\\
\leq &\left|\int_{t_0}^t \frac{1}{(\lambda c^j+ 1)} (H_j\left(\nu_{\lambda}(s)\right)- H_j(\nu_0(s)))ds\right|+\left|\int_{t_0}^t\frac{\lambda c^j}{\lambda c^j+1} H_j(\nu_0(s)) ds\right|\notag\\
= &  \frac{1}{(\lambda c^j+ 1)} \left|\int_{t_0}^t(H_j\left(\nu_{\lambda}(s)\right)- H_j(\nu_0(s)))ds\right|+\frac{\lambda c^j}{\lambda c^j+1}\left|\int_{t_0}^t H_j(\nu_0(s)) ds\right|\notag\\
= &  \frac{1}{(\lambda c^j+ 1)} \left|\int_{t_0}^t(H_j\left(Tw_{\lambda}(s)\right)- H_j(Tw_0(s)))ds\right|+\frac{\lambda c^j}{\lambda c^j+1}\left|\int_{t_0}^t H_j(\nu_0(s)) ds\right|
\end{align}  where $\xi\in\Omega_{h_0}$ and $H_j$ denotes the $j$-th coordinate of $H$. Since  $H(Tw_{\lambda})$ converges to $H(Tw_0)$ coordinate-wise as $Tw_{\lambda} $ converges to $Tw_0$ coordinate-wise with $\lambda\rightarrow 0^+$, uniformly with respect to $t\in\Omega_{h_0}$.
Letting $\lambda \rightarrow 0^+$   in (\ref{eqn-last}), we have for every $j\in\mathbb{N}$ and $t\in\Omega_{h_0}$,
 \begin{align}\label{eqn-last3}  \left|\int_{t_0}^t\left[(\lambda_n T+ I)^{-1} H\right]_j\left(\nu_{\lambda}(s)\right)ds-\int_{t_0}^t H_j(\nu_0(s))ds\right|\rightarrow 0 \mbox{ as $\lambda\rightarrow0^+$}.
 \end{align}
 By (\ref{eqn-last2}) and (\ref{eqn-last3}), we have for every   $t\in\Omega_{h_0}$,
 \[
 w_0(t)=w_{t_0}+\int_{t_0}^tH(\nu_0(s))ds,
 \]
 which combined with (\ref{eqn-last0}) gives
  \[
 w_0(t)=w_{t_0}+\int_{t_0}^tH(Tw_0(s))ds.
 \]
That is, $w_0$ is a solution of the initial value problem (\ref{Abs-ODE-IVP}). By analyticity of $w_0$, it is the unique solution of (\ref{Abs-ODE-IVP}) which is the complex extension of the real-valued solution $w=\left((y_1,\,z_1),\,(y_2,\,z_2),\,\cdots\right)\in C(\mathbb{R}; l_c^{\infty}(\mathbb{R}^{N+1}))$ at $t=t_0\in\mathbb{R}$. It follows that $(x,\,\tau)=(cy_1,\,cz_1)$ is analytic at $t_0$. Since $t_0\in\mathbb{R}$ is arbitrary, $(x,\,\tau)$ is analytic on $\mathbb{R}$. This completes the proof of Claim 5 and that of the theorem.\qed
 \end{proof}

\section{Example}\label{section-3}
In this section, we present an example from important applications. We now  study the analyticity of  periodic
solutions for the following delay differential equations with adaptive delay:
\begin{align}\label{ch3-eqn-4-23}
\left\{
\begin{aligned}
\dot{x}_1(t)&=-\mu x_1(t)+\sigma b( x_2(t-\tau(t))),\\
\dot{x}_2(t)&=-\mu x_2(t)+\sigma b( x_1(t-\tau(t))),\\
\dot{\tau}(t)&=1-h(x(t))\cdot(1+\tanh\tau(t)),
\end{aligned}
\right.
\end{align}
where $x(t)=(x_1(t),\,x_2(t))\in\mathbb{R}^2$, $\tau(t)\in\mathbb{R}$, $\tanh(\tau)= (e^{2\tau}-1)/(e^{2\tau}+1)$ and $\mu>0$ is a constant. 
We make the following assumptions:
\begin{description}
\item[$(\alpha_1)$] $b: \mathbb{R}\rightarrow\mathbb{R}$ and $h: \mathbb{R}^2\rightarrow\mathbb{R}$   are continuously differentiable functions with $b'(0)=-1$;
\item[$(\alpha_2)$] There exist  $h_0<h_1$ in $(1/2,\,1)$ such that $h_1> h(x)> h_0$ for all $x\in \mathbb{R}^2$;
\item[$(\alpha_3)$] $b$ is decreasing on $\mathbb{R}$ and the map $\mathbb{R}\ni y\rightarrow yb(y)\in\mathbb{R}$ is injective;
\item[$(\alpha_4)$]$yb(y)<0$ for $y\neq 0$, and there exists a continuous function $M:\mathbb{R}\ni \sigma\rightarrow M(\sigma)\in (0,\,+\infty)$ so that
\[
\frac{b(y)}{y}>-\frac{\mu}{2|\sigma|},
\]
for $|y|\geq M(\sigma)$;
 \item[$(\alpha_5)$] $h_0>(1+e^{-\pi})/2$   and there exists $\epsilon>0$ so that $b$ and $h$ have analytic complex extensions on
\[
 U_0\times V_0=\{(p,\,q)\in \mathbb{C}^2\times\mathbb{C}:  \Re(p,\,q)\in \overline{\Omega}_1,\, |\Im(p,\,q)|\leq\epsilon\}
\] where $\Omega_1= (-M(\sigma),\,M(\sigma))\times (-M(\sigma),\,M(\sigma))\times \left(0,\,-\frac{\ln (2h_0-1)}{2}\right).$
\end{description}
\begin{lemma}[\cite{HWZ}]\label{ch3-lemma-4-6}
Assume $(\alpha_1)$--$(\alpha_4)$ hold. Then  the range of every   periodic solution
$(x_1,\,x_2,\,\tau)$ of $($\ref{ch3-eqn-4-23}$\,)$ with $\sigma\in\mathbb{R}$ is contained in
\[
\Omega_1=(-M(\sigma),\,M(\sigma))\times(-M(\sigma),\,M(\sigma))\times \left(0, -\frac{\ln(2h_0-1)}{2}\right).
\]
\end{lemma}
\begin{theorem}
 Assume that $(\alpha_1)$--$(\alpha_5)$ hold. Then all the periodic solutions of (\ref{ch3-eqn-4-23}) are analytic on $\mathbb{R}$.
\end{theorem}
\begin{proof} By Lemma~\ref{ch3-lemma-4-6}, the range of every   periodic solution
$(x_1,\,x_2,\,\tau)$ of $($\ref{ch3-eqn-4-23}$\,)$ with $\sigma\in\mathbb{R}$ is contained in
$\Omega_1$. Now we apply Theorem~\ref{Analyticity-th}.  Let $l=1/2\in (0,\,1)$. For every $(x(t),\,\tau(t))=(x_1(t),\,x_2(t),\,\tau(t))\in \overline{\Omega}_1$, $t\in\mathbb{R}$, we have \[1\leq 1+\tanh \tau(t)\leq\frac{1}{h_0}<\frac{2}{1+e^{-\pi}}\] and  hence by $(\alpha_2)$ and $(\alpha_5)$ we obtain 
\begin{align*}
  & 1-(1-h(x(t))\cdot(1+\tanh\tau(t)))-\frac{e+l}{2}\\
= &\, h(x(t))\cdot(1+\tanh\tau(t)) -\frac{e+l}{2}\\
< &\, 1+\tanh \tau(t)-\frac{e+l}{2}\\
<  &\frac{2}{1+e^{-\pi}}-\frac{e+l}{2}\\
< & \, \frac{e-l}{2},
\intertext{and}
 & 1-(1-h(x(t))\cdot(1+\tanh\tau(t)))-\frac{e+l}{2}\\
= &\, h(x(t))\cdot(1+\tanh\tau(t)) -\frac{e+l}{2}\\
> &\, \frac{1+e^{-\pi}}{2}(1+\tanh \tau(t))-\frac{e+l}{2}\\
\geq  &\frac{1+e^{-\pi}}{2}-\frac{e+l}{2}\\
>& \, -\frac{e-l}{2}.
\end{align*}Therefore, we have $| 1-(1-h(x(t))\cdot(1+\tanh\tau(t)))-\frac{e+l}{2}|<\frac{e-l}{2}$ for all $(x(t),\,\tau(t))\in\overline{\Omega}_1$. Note that $1>h_0>(1+e^{-\pi})/2$ and $(x(t),\,\tau(t))\in\overline{\Omega}_1$ imply that $0< \tau(t)<\frac{\pi}{2}$. And the complex extension of $1+\tanh q$ is analytic for $|q|<\frac{\pi}{2},\,q\in \mathbb{C}$. Then by $(\alpha_5)$ we can choose $\epsilon_0\in (0,\,\epsilon)$
 small enough so that $| 1-(1-h(p)\cdot(1+\tanh q))-\frac{e+l}{2}|<\frac{e-l}{2}$ for all $(p,\,q)\in   U\times V$ where 
\begin{align*}
 U\times V= \{(p,\,q)\in \mathbb{C}^2\times\mathbb{C}:  \Re(p,\,q)\in  {\Omega}_1,\, |\Im(p,\,q)|<\epsilon_0\}\subset U_0\times V_0.
\end{align*}
Then by applying Theorem~\ref{Analyticity-th} on $U\times V$,  analyticity of all the periodic solutions of (\ref{ch3-eqn-4-23}) follows.
\hfill\hspace*{1em}\qed
\end{proof}

%

\end{document}